\documentclass[11pt]{amsart}
\usepackage{xcolor}
\usepackage[margin=1in]{geometry}
\usepackage{amsmath,amsthm,mathtools}
\usepackage[shortlabels]{enumitem}
\usepackage{setspace}
\usepackage{hyperref}
\usepackage{todonotes}

\usepackage[no-math]{fontspec}
\usepackage[cachedir=minted-cache]{minted}

\newmintinline[lean]{lean4}{bgcolor=white}
\newminted[leancode]{lean4}{escapeinside=!!, 
                            breaklines, 
                            fontsize=\normalsize}
\numberwithin{equation}{section}

\newtheorem{theorem}{Theorem}[section]
\newtheorem{proposition}[theorem]{Proposition}
\newtheorem{lemma}[theorem]{Lemma}

\theoremstyle{definition}
\newtheorem{definition}[theorem]{Definition}
\theoremstyle{remark}
\newtheorem{remark}[theorem]{Remark}
\newtheorem{question}[theorem]{Question}

\newcommand{\E}{\mathbb E}
\newcommand{\Prob}{\mathbb P}
\newcommand{\N}{\mathbb N}
\newcommand{\R}{\mathbb R}
\newcommand{\Span}{\operatorname{span}}
\newcommand{\supp}{\operatorname{supp}}
\newcommand{\sgn}{\operatorname{sign}}

\newcommand{\eps}{\varepsilon}
\newcommand{\abs}[1]{\lvert #1 \rvert}
\newcommand{\norm}[1]{\lVert #1 \rVert}
\newcommand{\ip}[2]{\langle #1,#2\rangle}

\title[Stable Phase Retrieval for Independent Random Variables]{Stable Phase Retrieval for Spans of Independent Random Variables}

\author{Pedro Abdalla\textsuperscript{1}}
\author{Jaume de Dios Pont\textsuperscript{2}}
\author{Jo\~ao P. G. Ramos\textsuperscript{3}}
\author{Mitchell A. Taylor\textsuperscript{4}}

\begin{document}
\setstretch{1.2}

\begin{abstract} We prove that, after $L^2$ normalization, stable phase retrieval
holds over the $L^2$-spans of independent real-valued centered random variables if and only if all but possibly one coordinate satisfies a uniform two-sided
$L^1$ bound. This provides a complete characterization of stable phase retrieval for such subspaces, building upon the pioneering work of Calderbank--Daubechies--Freeman--Freeman and confirming the conjectured characterization communicated to us by those authors.

We provide two different proofs of this fact, both based on a decomposition of the $\ell^2$-coefficients of each random variable.  The first is a compactness proof, which makes use of the infinite divisibility of limit laws of tail sums. The second is a quantitative proof, which substitutes the compactness step with an explicit dichotomy based on anticoncentration estimates of Sperner type. This latter proof was partially LLM generated based on the ideas in the first proof and a considerable amount of guidance by the authors. An autoformalization of our main result in Lean 4 is also provided, following the ideas in the quantitative proof. 
\end{abstract}

\maketitle
\begingroup
\renewcommand{\thefootnote}{\arabic{footnote}}
\footnotetext[1]{Department of Mathematics, University of California, Irvine, Irvine, California 92697, USA}
\footnotetext[2]{Center for Data Science, New York University, New York, New York 10011, USA}
\footnotetext[3]{Instituto de Matem\'atica Pura e Aplicada (IMPA), Rio de Janeiro, RJ, Brazil}
\footnotetext[4]{Department of Mathematics, ETH Z\"urich, Z\"urich, Switzerland}
\endgroup
\setcounter{footnote}{0}

\section{Introduction}

\subsection{Historical background} Phase retrieval is a class of inverse problems in which the measurement device captures only magnitude information, while the phase remains unobserved. It dates back at least as far as the seminal Pauli problem in quantum mechanics, which asks to what extent information about a wave function can be encoded in the magnitudes of the function and its Fourier transform \cite{MR3552875,MR3256781,pauli1933allgemeinen,ramos-sousa-pauli}. It also arises naturally in optics and signal processing, where detectors can typically only record  the  intensity of the signal \cite{MR3069958,eldar2014stability,grohs2020survey,sanz1984phase}. All of these problems can be modeled as the recovery (up to acceptable ambiguities) of a vector $x$ from the phaseless measurement $|Tx|$, where $T: H \to X$ is an injective linear mapping from a signal space $H$ to a function space $X$. Such problems can be equivalently formulated as the recovery of a function $f\in T(H)$ from its magnitude $|f|\in X$, again up to acceptable ambiguities. In the real-valued setting, the unavoidable
ambiguity is multiplication by a global sign, and the natural question is whether
this is the only ambiguity.  

The finite-dimensional phase retrieval problem has been extensively studied, both for random and structured measurements. The fundamental tasks of developing injectivity and stability criteria have been thoroughly investigated, and practical tools such as polynomial-time algorithms to solve the phase retrieval problem
have been developed in special cases \cite{bandeira2014saving,MR3260258,MR3069958,DB,eldar2014stability,MR3746047,KS}.
In contrast, the infinite-dimensional theory is far more delicate and underdeveloped. In particular, although injectivity may hold in many situations of interest \cite{cahill2016infinite,freeman2025cahill}, instability is often unavoidable. Most notably, uniform stability cannot hold for continuous
frames over infinite-dimensional spaces \cite{AG-continuous-frames}, so canonical structured systems
such as the Gabor transform induce ill-posed phase retrieval problems
\cite{alaifari2021gabor,alaifari2025cheeger,grohs2021stable,steinerberger2020stability}. For general background on the phase retrieval problem for frames, we refer the reader to \cite{grohs2020survey}.

Mathematically, the uniqueness problem for phase retrieval asks whether the nonlinear mapping $f\mapsto |f|$ is injective on the image of the signal space, after quotienting by the natural global phase ambiguity. The stability problem asks whether the inverse of this map obeys a Lipschitz/H\"older  estimate, at least locally. Although both questions are relevant for applications, stability is indispensable for any feasible reconstruction algorithm, as injectivity alone is
too weak for numerical reconstruction.

One of the main conceptual contributions of the recent series of papers
\cite{calderbank2022stable,camunez2025characterization,christ2022examples,FOTP,garcia2025isometric,garcia2025existence} is to make the stability problem for phase retrieval intrinsic to the geometry of the
subspace rather than to a particular measurement operator $T$. This viewpoint permits a unified theory, where stability in infinite dimensions often holds.   More formally, we have the following definition.
\begin{definition}\label{DefinitionSPR}
    Let $(X,\|\cdot\|_X)$ be a real normed function space and let
$E\subset X$ be a subspace.  We say that $E$ does \emph{stable phase
retrieval} if there exists a constant $C<\infty$ such that
\begin{equation*}
\label{def:stable_pr}
\min_{\sigma\in\{\pm1\}}\norm{f-\sigma g}_X
\le C\,\norm{\abs{f}-\abs{g}}_X
\qquad\text{for all }f,g\in E.
\end{equation*}
\end{definition}

\subsection{Main result} In this article, we will be concerned with the quantitative aspects of phase retrieval for random subspaces $E$ spanned
by independent random variables in $L^2$. 
In this setting, Calderbank, Daubechies, Freeman and Freeman
\cite{calderbank2022stable} proved stable phase retrieval in a simplified model with disjoint indicator functions attached to each random variable. They  then proposed the following conjecture and communicated it to us, which would give a complete characterization of independent sequences doing stable phase retrieval in $L^2$: If $(\xi_i)$ is an orthonormal sequence of independent, centered
random variables in $L^2$, then $\overline{\Span}(\xi_i)$ does stable phase retrieval if and only if there exist constants $0<A\le B<1$ such that $A\le \norm{\xi_i}_{L^1}\le B$ for all but at most one index. 

The necessity of the two-sided uniform bound on the $L^1$ norms of the $\xi_i$ is standard. Indeed, the fact that $A$ is bounded away from zero prevents one from constructing an almost disjoint sequence in $\overline{\Span}(\xi_i)$. Clearly, disjointly supported functions will obstruct phase retrieval \cite{FOTP}, as if $d_1$ and $d_2$ are pairwise disjoint, then $f:=d_1+d_2$ and $g:=d_1-d_2$ have the same absolute value but are not multiples of each other. On the other hand, the condition $B<1$ excludes the Rademacher
random variables (random signs) which satisfy $|\xi_i|=1$ almost surely and therefore also violate phase retrieval.
The restriction to “at most one exceptional coordinate” is also sharp, as one Rademacher random variable will not cause issues, but two independent Rademachers  will immediately violate phase retrieval. 

The core question is whether these conditions are also sufficient for stable retrieval to hold. Christ, Pineau and the last-named author \cite{christ2022examples} proved that they suffice under the additional assumptions of finite fourth moment and identical distribution. 
The main theorem of this manuscript proves the above conjecture in full generality, hence  characterizing when a subspace spanned by independent, centered, real-valued random variables does stable phase retrieval in $L^2$. 
\begin{theorem}[Stable phase retrieval for independent random variables]\label{thm:main}
Let $\eta,\xi_2,\xi_3,\dots$ be independent real-valued random variables and
fix $\delta>0$.  Assume that
\[
\norm{\eta}_{L^2}=1,
\qquad
\E[\eta]=0,
\]
and, for every $i\ge2$,
\[
\norm{\xi_i}_{L^2}=1,
\qquad
\E[\xi_i]=0,
\qquad
\delta\le \norm{\xi_i}_{L^1}\le 1-\delta.
\]
Then there exists $C_{\delta,\eta}\ge1$ such that, for all
$X,Y\in \overline{\Span}(\eta,\xi_i:i\ge2)$,
\[
\min_{\sigma\in\{\pm1\}}\norm{X-\sigma Y}_{L^2}
\le
C_{\delta,\eta}\,\norm{\abs{X}-\abs{Y}}_{L^2}.
\]
\end{theorem}

We present two proofs of this result. The first is a qualitative and conceptually enlightening proof, while the second yields quantitative estimates for $C_{\delta,\eta}$. A common starting point for both proofs is 
a theorem of Freeman, Oikhberg, Pineau and the last-named author \cite{FOTP}, which identifies orthogonal vectors as the main source of instability for phase retrieval. More precisely, for subspaces of $L^2$, it suffices to prove stable phase retrieval for orthogonal pairs $f,g$ with $\norm{f}=1$ and $\norm{g}
\leq 1$.

Next, we explain the main ideas behind our proofs. In the first proof we assume, for the sake of a contradiction, that the result does not hold.  We consider a  decomposition of the coefficients of the alleged pairs of random variables $(\{X_n\}_n, \{Y_n\}_n)$ in the span of $\{\eta,\xi_2,\xi_3,\dots\}$ that generate a counterexample  by splitting them into two components. After rearranging the coordinates, the first component consists of the coordinates that converge strongly in the $\ell^2$ sense. We refer to this as the ``head" component. The remaining coordinates may be chosen to be diffuse, i.e., their $\ell^{\infty}$-norms vanish,  constituting the ``tail" component. 

To bypass the obstacle that the diffuse coordinates may not converge strongly, we rely on characterizations of weak limits of triangular arrays. By independence between the coordinates, we may condition on the head component without affecting the law of the tail component. More accurately, we may derive an identity in the limit for the pair $(\{X_n\}_n, \{Y_n\}_n)$ of the form $|H_X + R_X| = |H_Y + R_Y|,$ where $H_X = \tau H_Y$ for some universal sign $\tau \in \{\pm1\},$ and $(R_X,R_Y)$ is an infinitely divisible random vector that is  independent of $(H_X,H_Y)$. From this equality, it is relatively straightforward to reach a contradiction by conditioning on $(H_X,H_Y)$ and invoking geometric properties of infinitely divisible laws.

Our second proof focuses on (crude) quantitative estimates for the constant $C_{\delta,\eta}$ in Theorem \ref{thm:main}. To achieve this, it is necessary  to modify the argument involving infinitely divisible laws, as they only appear in the limit as $n$ diverges. To this end, the second proof keeps the same head--tail decomposition as the first proof, but it replaces the limit law of triangular arrays  by an
explicit concentration versus diffusion dichotomy for the tail part. A key extra ingredient is an anticoncentration estimate, which acts as a quantitative version of the fact that infinitely divisible laws cannot concentrate on two different lines. This latter geometric observation is a crucial fact exploited in the first proof and, for this reason, the two proofs address
the same theorem from complementary angles. More precisely, the compactness proof shows that a
bad sequence cannot exist by passing to a limiting model, while the quantitative proof
shows directly that every orthogonal pair already contains a definite amount of
modulus separation.

\subsection{Roadmap}
The rest of this manuscript is organized as follows.   Section~\ref{sec:preliminaries} collects
the basic definitions and tools used throughout the manuscript. 
Section~\ref{sec:compactness-proof}  presents the compactness proof of
Theorem~\ref{thm:main}.  Section~\ref{sec:quant-proof} gives a second proof of
Theorem~\ref{thm:main}, explaining how the asymptotic step can be bypassed by an
explicit decomposition together with an anticoncentration argument.
Section~\ref{sec:comments} discusses extensions, related directions
and open problems. The article then concludes with a discussion of the Lean formalization.  

\section*{Usage of large language models} 

Large language models played a significant role in the development of this work. Our starting point was a fully human proof of an $L^p$ version of Theorem~\ref{thm:main}  for $p> 2$, based on compactness and central limit ideas. Our goal, however, was the endpoint $L^2$ statement. By adapting arguments proposed by us and then found by GPT 5.4 in the literature, we managed to fine-tune our proof to achieve the endpoint $L^2$ result, by leveraging a non-trivial Gaussian part of the limiting infinitely divisible distributions (see Section \ref{ssec:gaussian-alt} for more details). 

A first key contribution of an LLM was the suggestion of the general principle that one may forego with the Gaussian nature of the limit and replace it with the fact that a nontrivial infinitely divisible random vector in $\R^2$ supported on the cross 
$$\{ (x,y) \in \R^2 \colon |x| = |y|\}$$
must, in fact, be supported on one of the two diagonals $\{x=y\}$ or $\{x=-y\}.$ This observation later became one of our conceptual pillars, as it was something that we had not fully appreciated when writing the first proof of our main result.

A second key contribution came at the quantitative level. Seeking an argument that was both explicitly quantitative and better suited  to formal verification, we prompted LLMs for an alternative to the infinitely divisible proof. The resulting strategy -- LLM-generated with significant human guidance and access to our original proof -- was combinatorial, of Sperner/antichain type, and yielded a quantitative (and arguably technically simpler) version of our main result, showing that under the appropriate hypotheses, closeness to the cross forces closeness to one of the two diagonals. This quantitative argument was then formalized and machine-checked in Lean. 

In this project, GPT 5.4 was used mainly for mathematical exploration, and Claude Opus 4.6 for assistance with the Lean formalization. Later versions of these models were also employed at the very end of the project to clean up the Lean formalization and to check for grammatical errors in the paper. The authors independently checked all of the final statements and  substantially rewrote the proofs. 
\section*{Acknowledgments} 
P.A.~was supported by the NSF and by the Simons Collaboration on Theoretical Foundations of Deep Learning. J.P.G.R.~was supported by the FCT through project SHADE (project 2023.17881.ICDT, DOI  10.54499 / 2023.17881.ICDT) and by FAPERJ through the JCNE grant no.~SEI-260003/020475/2025. Several of the main ideas in this paper were conceived when P.A., J.P.G.R.~and M.A.T.~participated in the research term ``Lattice Structures in Analysis and Applications" at ICMAT, Madrid, in May 2024.

\section{Preliminaries}\label{sec:preliminaries}

\subsection{Stable phase retrieval and the orthogonal reduction} Throughout the article, we will use the formulation of stable phase retrieval in Definition~\ref{DefinitionSPR}. A key tool for proving stable phase retrieval is the orthogonal reduction from \cite[Theorem~3.9]{FOTP} (see also~\cite[Theorem~2.1]{localphase} for a similar reduction for frames, which does not apply here). In \cite[Theorem~3.9]{FOTP} the reduction is stated in the general setting of Banach lattices. Here, we state the result in the simplified setting of $L^2(\mu)$. 


\begin{proposition}[Orthogonal reduction]\label{prop:orth-reduction}
Let $E\subset L^2(\mu;\mathbb{R})$ be a subspace. Then $E$
does stable phase retrieval if and only if there exists $c>0$ such that
\[
\norm{\abs{f}-\abs{g}}_{L^2}\ge c
\]
for every orthogonal pair $f,g\in E$ with $\norm{f}_{L^2}=\norm{g}_{L^2}=1$. 
\end{proposition}

\begin{proof} The result from \cite[Theorem~3.9]{FOTP} shows that, in order to verify stable phase retrieval on a subspace, it suffices to verify it for orthogonal pairs $f$ and $g$ satisfying $\|f\|_{L^2}=1$ and $\|g\|_{L^2}\le 1$. Suppose then that the assertion of Proposition \ref{prop:orth-reduction} holds, but stable phase retrieval does not. Then, there exists a sequence of pairs $\{(f_n,g_n)\}_{n \in \N}$ satisfying $\|f_n\|_{L^2} = 1 \geq \|g_n\|_{L^2}$ and $f_n \perp g_n$ such that
\[
\limsup_{n \to \infty} \| |f_n| - |g_n| \|_{L^2} = 0.
\]
We first claim that $\|g_n\|_{L^2} \to 1.$ Indeed, if not, then passing to a subsequence, we would be able to estimate 
\[
\||f_n| - |g_n|\|_{L^2} \geq \|f_n\|_{L^2} - \|g_n\|_{L^2} \ge \delta
\]
for some fixed $\delta>0$. This contradicts the fact that the pair $(f_n,g_n)$ violates stable phase retrieval asymptotically. Next,  we note that
\begin{equation}\label{eq:comparison}
\||f_n| - |g_n|\|_{L^2} \geq \| |f_n| - \lambda_n |g_n|\|_{L^2} - \frac{|\lambda_n - 1|}{\lambda_n},
\end{equation}
where $\lambda_n := \frac{1}{\|g_n\|_{L^2}}.$ By hypothesis, the first summand on the right-hand side of \eqref{eq:comparison} is bounded from below by $c$, while the term $\frac{|\lambda_n - 1|}{\lambda_n}$ vanishes because $\|g_n\|_{L^2}$ converges to one. Hence, we again reach a contradiction, which shows that the reduction we were after is indeed valid. 
\end{proof}

\subsection{Test functions and \texorpdfstring{$L^1$}{L1} bounds}
An important ingredient in our proof is the existence of appropriate test functions, which we will use to extract coefficients from our random variables. We will refer to these test functions as ``probes", as they will allow us to indirectly observe various structural properties of our random variables.
\begin{lemma}\label{lem:probes}
Let $\xi$ be a real-valued random variable with $\E[\xi]=0$ and $\E[\xi^2]=1$.

\begin{enumerate}[label=\rm(\arabic*)]
\item If $\E[\abs{\xi}]\ge A$ for some constant $A>0$, then there exists a bounded measurable
function $S$ such that
\[
\E[S]=0,
\qquad
\E[\xi S]=1,
\qquad
\norm{S}_{L^\infty}\le \frac{2}{A}.
\]
\item If $\E[\abs{\xi}]\le B$ for some constant $B<1$, then there exists a bounded measurable
function $T$ such that
\[
\E[T]=0,
\qquad
\E[(\xi^2-1)T]=1,
\qquad
\norm{T}_{L^\infty}\le \frac{2}{1-B}.
\]
\end{enumerate}
\end{lemma}

\begin{proof}
We first construct the function with the properties described in part (1).  Define
\[
S_0:=\sgn(\xi),
\]
with the convention $\sgn(0)=0$.  Note that $S_0$ is measurable and satisfies
$|S_0|\le 1$ almost surely.  Moreover,
\[
\xi S_0=\abs{\xi}
\qquad\text{almost surely,}
\]
so
\[
\E[\xi S_0]=\E[\abs{\xi}]\ge A>0.
\]
Using this property, we may normalize and set
\[
S:=\frac{S_0-\E[S_0]}{\E[\xi S_0]}.
\]
The function $S$ is measurable and bounded. In addition, we have
\[
\E[S]=\frac{\E[S_0]-\E[S_0]}{\E[\xi S_0]}=0,
\]
and, using the fact that $\E[\xi]=0$, we also have
\[
\E[\xi S]
=\frac{\E[\xi S_0]-\E[S_0]\E[\xi]}{\E[\xi S_0]}
=\frac{\E[\xi S_0]}{\E[\xi S_0]}
=1.
\]
Finally, since $|S_0|\le1$ and $|\E[S_0]|\le \E|S_0|\le1$, it follows that
\[
|S_0-\E[S_0]|\le 2
\qquad\text{almost surely.}
\]
Therefore,
\[
\norm{S}_{L^\infty}
\le \frac{2}{\E[\xi S_0]}
=\frac{2}{\E[\abs{\xi}]}
\le \frac{2}{A}.
\]
This proves part~\rm(1).

To construct the function $T$ in statement (2), we apply the same idea to the centered random variable
\[
\eta:=\xi^2-1.
\]
Since $\E[\xi^2]=1$, we have $\E[\eta]=0$.  We now show that $\eta$ has
 $L^1$ mass bounded away from zero.  Since $\xi^2=\abs{\xi}^2$, we may compute that
\[
\abs{\eta}
=\abs{\abs{\xi}^2-1}
=\abs{\abs{\xi}-1}(\abs{\xi}+1)
\ge \abs{\abs{\xi}-1}.
\]
Taking expectations and using the reverse triangle inequality, we obtain
\[
\E[\abs{\eta}]
\ge \E[\abs{\abs{\xi}-1}]
\ge \bigl|\E[\abs{\xi}]-1\bigr|
=1-\E[\abs{\xi}]
\ge 1-B.
\]
Now, we set
\[
T_0:=\sgn(\eta)
\]
and define
\[
T:=\frac{T_0-\E[T_0]}{\E[\eta T_0]}
=\frac{T_0-\E[T_0]}{\E[\abs{\eta}]}.
\]
Exactly as in part (1), we have that $T$ is measurable, bounded and satisfies
\[
\E[T]=0,
\qquad
\E[\eta T]=1.
\]
Since $\eta=\xi^2-1$, we have
\[
\E[(\xi^2-1)T]=1.
\]
The $L^\infty$ bound 
\[
\norm{T}_{L^\infty}
\le \frac{2}{\E[\abs{\eta}]}
\le \frac{2}{1-B}
\]
is straightforward, so the proof is complete.
\end{proof}

\begin{remark}\label{rem:probe-extraction}
If $\xi_1,\dots,\xi_n$ are independent centered random variables and
$X_c=\sum_i c_i\xi_i$, then we can use the test functions constructed above to eliminate every
coordinate except the one being tested.  More precisely,
\[
\E[X_c\,S_i]=c_i,
\qquad
\E[X_c^2\,T_i]=c_i^2,
\qquad
\E[X_c^2\,S_iS_j]=2c_ic_j
\quad(i\ne j),
\]
whenever $S_i,S_j,T_i$ are the corresponding probes for $\xi_i,\xi_j$.
\end{remark}

The next basic estimate  will allow us to transfer the lower bound for the $L^1$ norms of the random variables to their entire span (with a slightly different constant). More accurately, we prove the following result.

\begin{lemma}[$L^1$ lower bound for independent sums]\label{lem:l1lower}
Let $\zeta_1,\dots,\zeta_N$ be independent, centered, real-valued random variables and
assume that
\[
\norm{\zeta_i}_{L^1}\ge \delta \norm{\zeta_i}_{L^2}>0
\qquad(1\le i\le N).
\]
Then for every scalar family $(c_i)_{i=1}^N$, we have
\[
\norm{\sum_{i=1}^N c_i\zeta_i}_{L^1}
\ge
\frac{\delta}{2\sqrt2}
\Bigl(\sum_{i=1}^N c_i^2\norm{\zeta_i}_{L^2}^2\Bigr)^{1/2}.
\]
\end{lemma}

\begin{proof}
The statement is homogeneous in the coefficient vector, so we may first
normalize and assume that
\[
\sum_{i=1}^N c_i^2\norm{\zeta_i}_{L^2}^2=1.
\]
We next reduce to the unit-variance case.  For each $i$, set
\[
\widetilde\zeta_i:=\frac{\zeta_i}{\norm{\zeta_i}_{L^2}},
\qquad
\widetilde c_i:=c_i\norm{\zeta_i}_{L^2}.
\]
The variables $\widetilde\zeta_i$ are still independent and centered. Moreover, they satisfy $\norm{\widetilde\zeta_i}_{L^2}=1$ and the lower $L^1$ bound
becomes
\[
\norm{\widetilde\zeta_i}_{L^1}
=\frac{\norm{\zeta_i}_{L^1}}{\norm{\zeta_i}_{L^2}}
\ge \delta.
\]
We also have
\[
\sum_{i=1}^N \widetilde c_i \widetilde\zeta_i
=\sum_{i=1}^N c_i\zeta_i,
\qquad
\sum_{i=1}^N \widetilde c_i^2
=1.
\]
Thus, it is enough to prove the lemma under the additional assumptions
\[
\norm{\zeta_i}_{L^2}=1
\qquad(1\le i\le N),
\qquad
\sum_{i=1}^N c_i^2=1.
\]
Define
\[
S:=\sum_{i=1}^N c_i\zeta_i
\]
and let
\[
S':=\sum_{i=1}^N c_i\zeta_i'
\]
be an independent copy of $S$, obtained from an independent copy
$(\zeta_i')_{i=1}^N$ of the family $(\zeta_i)_{i=1}^N$.
By the triangle inequality, we have
\[
\abs{S-S'}\le \abs{S}+\abs{S'}.
\]
Taking expectations and using that $S$ and $S'$ have the same distribution, we
obtain
\[
\norm{S-S'}_{L^1}\le 2\norm{S}_{L^1}.
\]
Now, the random variables $\zeta_i-\zeta_i'$ are independent and symmetric.
Conditioning on their absolute values and applying the Khintchine inequality
gives
\[
\norm{S-S'}_{L^1}
=
\Bigl\|\sum_{i=1}^N c_i(\zeta_i-\zeta_i')\Bigr\|_{L^1}
\ge
\frac{1}{\sqrt2}\,
\E\Bigl[\Bigl(\sum_{i=1}^N c_i^2(\zeta_i-\zeta_i')^2\Bigr)^{1/2}\Bigr].
\]
Since the nonnegative numbers $(c_i^2)_{i=1}^N$ sum to one, Cauchy--Schwarz
implies the pointwise estimate
\[
\sum_{i=1}^N c_i^2\abs{\zeta_i-\zeta_i'}
\le
\Bigl(\sum_{i=1}^N c_i^2\Bigr)^{1/2}
\Bigl(\sum_{i=1}^N c_i^2(\zeta_i-\zeta_i')^2\Bigr)^{1/2}
=
\Bigl(\sum_{i=1}^N c_i^2(\zeta_i-\zeta_i')^2\Bigr)^{1/2}.
\]
Taking expectations and combining with the previous inequality, we arrive at
\[
\norm{S-S'}_{L^1}
\ge
\frac{1}{\sqrt2}\sum_{i=1}^N c_i^2\norm{\zeta_i-\zeta_i'}_{L^1}.
\]
Hence,
\[
\norm{S}_{L^1}
\ge
\frac{1}{2\sqrt2}\sum_{i=1}^N c_i^2\norm{\zeta_i-\zeta_i'}_{L^1}.
\]
It remains to bound each $\norm{\zeta_i-\zeta_i'}_{L^1}$ from below.  Since
$\zeta_i'$ is independent of $\zeta_i$ and centered, we have
\[
\E[\zeta_i-\zeta_i' \mid \zeta_i]
=\zeta_i.
\]
Applying Jensen's inequality to the conditional expectation gives
\[
\norm{\zeta_i-\zeta_i'}_{L^1}
\ge
\norm{\E[\zeta_i-\zeta_i' \mid \zeta_i]}_{L^1}
=\norm{\zeta_i}_{L^1}
\ge \delta.
\]
Substituting this into the previous bound yields
\[
\norm{S}_{L^1}
\ge
\frac{\delta}{2\sqrt2}\sum_{i=1}^N c_i^2
=\frac{\delta}{2\sqrt2}.
\]
Since this is exactly the normalized form of the desired estimate, the proof is
complete.
\end{proof}

\subsection{Triangular arrays and infinitely divisible limits}
A large component of our proof focuses on understanding the weak limits of the random variables we are analyzing. The next lemma is a precise statement of this type.

\begin{lemma}\label{lem:tail-rigidity}
Let $(\xi_{n,i})_{n\ge1,\ i\ge2}$ be a triangular array such that for each
fixed $n$ the family $(\xi_{n,i})_{i\ge2}$ is independent, centered,
$L^2$-normalized and satisfies
\[
\norm{\xi_{n,i}}_{L^1}\ge \delta
\qquad(i\ge2).
\]
Let $c^{(n)}\in \ell^2(\N)$ satisfy
\[
\sup_n \norm{c^{(n)}}_2<\infty,
\qquad
c_1^{(n)}=0,
\qquad
\max_{i\ge2}\abs{c_i^{(n)}}\to0.
\]
Set
\[
Z_n:=\sum_{i\ge2} c_i^{(n)}\xi_{n,i}.
\]
If $Z_n\to0$ in distribution, then $\norm{c^{(n)}}_2\to0$.
\end{lemma}

\begin{proof}
Since the limiting law is the Dirac mass at $0$, the assumption
$Z_n\to0$ in distribution implies that
\[
Z_n\to0
\qquad\text{in probability.}
\]
To pass from convergence in probability to convergence of first moments, we use
uniform integrability.

For each $n$, the variables $(\xi_{n,i})_{i\ge2}$ are independent, centered
and $L^2$-normalized.  Therefore, they form an orthonormal family in $L^2$, and
the random sum $Z_n$ satisfies
\[
\norm{Z_n}_{L^2}^2
=\E\Bigl[\Bigl(\sum_{i\ge2} c_i^{(n)}\xi_{n,i}\Bigr)^2\Bigr]
=\sum_{i\ge2}(c_i^{(n)})^2\E[\xi_{n,i}^2]
=\sum_{i\ge2}(c_i^{(n)})^2
=\norm{c^{(n)}}_2^2.
\]
Since $\sup_n \norm{c^{(n)}}_2<\infty$, the sequence $(Z_n)_n$ is bounded in
$L^2$.  Any $L^2$-bounded family is uniformly integrable in $L^1$, so
$(|Z_n|)_n$ is uniformly integrable.  Since $Z_n\to0$ in probability, we may
therefore conclude that
\[
\norm{Z_n}_{L^1}\longrightarrow0.
\]
We now apply Lemma~\ref{lem:l1lower} to each row of the array.  The assumptions of
that lemma are satisfied because every $\xi_{n,i}$ is centered,
$L^2$-normalized and obeys the bound $\norm{\xi_{n,i}}_{L^1}\ge\delta$.  We conclude that
\[
\norm{Z_n}_{L^1}
=\Bigl\|\sum_{i\ge2} c_i^{(n)}\xi_{n,i}\Bigr\|_{L^1}
\ge
\frac{\delta}{2\sqrt2}
\Bigl(\sum_{i\ge2}(c_i^{(n)})^2\Bigr)^{1/2}
=\frac{\delta}{2\sqrt2}\norm{c^{(n)}}_2.
\]
The left-hand side tends to zero, so the right-hand side must also tend to
zero.  Therefore,
\[
\norm{c^{(n)}}_2\longrightarrow0,
\]
which is the desired conclusion.
\end{proof}
In order to state our next result, we recall that an $\R^n$-valued random variable $S$ is \emph{infinitely divisible} if, for each positive integer $m$, there exist i.i.d.~random variables $\eta_{m,1},\eta_{m,2},\dots,\eta_{m,m}$ such that 
\[
S \overset{d}{\sim} \eta_{m,1} + \cdots + \eta_{m,m}. 
\]
Infinitely divisible distributions are a fundamental object in probability theory due to the classical limit theorem of triangular arrays (see, for instance, \cite[Chapter~IV]{gnedenko-kolmogorov}).
Our next preliminary result shows that the weak limit of diffuse random vectors is an infinitely divisible random variable. 

\begin{lemma}[Diffuse limits are infinitely divisible]\label{lem:diffuse-id-limit}
Let $d\ge1$ and let $(\eta_{n,i})_{n\ge1,\ i\ge1}$ be a triangular array of
$\R^d$-valued random variables.  Assume that, for each fixed $n$, the variables
$(\eta_{n,i})_{i\ge1}$ are independent and that the row sum
\[
S_n:=\sum_{i\ge1} \eta_{n,i}
\]
is well-defined as a limit in probability of its finite partial sums.  Assume
also that the array is infinitesimal, in the sense that, for every
$\eps>0$,
\[
\sup_{i\ge1}\Prob\{\norm{\eta_{n,i}}>\eps\}\longrightarrow0.
\]
If $S_n$ converges in distribution to an $\R^d$-valued random variable $S$,
then the law of $S$ is infinitely divisible.
\end{lemma}
\begin{proof}
The version with finitely many summands in each rows is classical. More precisely, if 
$U_{n,1},\dots,U_{n,k_n}$ are independent $\R^d$-valued random variables,
\[
\max_{1\le j\le k_n}\Prob\{\norm{U_{n,j}}>\eps\}\longrightarrow0
\qquad(\eps>0),
\]
and
\[
\sum_{j=1}^{k_n}U_{n,j}\Rightarrow U,
\]
then the law of $U$ is infinitely divisible.  This is the celebrated
Kolmogorov--Khintchine theorem for null arrays
\cite[Theorem~9.3]{sato-levy}.  In the one-dimensional case, it is also the
classical theorem of Gnedenko--Kolmogorov, as can be found, for example,  in
\cite[Chapter~IV, Section~24]{gnedenko-kolmogorov}.

It remains only to reduce the present formulation with countably many summands per row to the
finite case.  Since $S_n$ is the limit in probability of its finite
partial sums, for each $n$ we may choose an integer $k_n$ such that
\[
\Prob\left\{
\norm{S_n-\sum_{i=1}^{k_n}\eta_{n,i}}>\frac1n
\right\}\le \frac1n.
\]
Set
\[
S_n^{(k)}:=\sum_{i=1}^{k_n}\eta_{n,i}.
\]
Then $S_n^{(k)}-S_n\to0$ in probability and  hence, by Slutsky's theorem,
\[
S_n^{(k)} = (S_n^{(k)}-S_n)+S_n\Rightarrow S.
\]
For each fixed $n$, the summands $\eta_{n,1},\dots,\eta_{n,k_n}$ are independent,
and the finite array remains infinitesimal because
\[
\max_{1\le i\le k_n}\Prob\{\norm{\eta_{n,i}}>\eps\}
\le
\sup_{i\ge1}\Prob\{\norm{\eta_{n,i}}>\eps\}
\longrightarrow0.
\]
Applying the classical theorem, we conclude that the law of $S$ is infinitely
divisible.
\end{proof}

\section{The compactness proof}\label{sec:compactness-proof}

This section proves Theorem~\ref{thm:main} through the first method mentioned in the introduction. It consists of an argument by contradiction that focuses on extracting limit properties of sequences failing stable phase retrieval. 

\subsection{Profile decomposition} We start with the basic setup that can be deduced from the failure of stable phase retrieval. 

\begin{proposition}\label{prop:counterexample}
If the conclusion of Theorem~\ref{thm:main} is false, then there exists
\begin{enumerate}[label=\rm(\roman*)]
\item a triangular array $(\xi_{n,i})_{n\ge1,\ i\ge2}$ such that for each
fixed $n$ the family $(\eta,\xi_{n,i})_{i\ge2}$ is independent and every
$\xi_{n,i}$ satisfies
\[
\norm{\xi_{n,i}}_{L^2}=1,
\qquad
\E[\xi_{n,i}]=0,
\qquad
\delta\le \norm{\xi_{n,i}}_{L^1}\le 1-\delta;
\]
\item coefficient vectors $a^{(n)},b^{(n)}\in \ell^2(\N)$ with
\[
\norm{a^{(n)}}_2=\norm{b^{(n)}}_2=1,
\qquad
\ip{a^{(n)}}{b^{(n)}}=0,
\]
such that, writing
\[
X_n:=a_1^{(n)}\eta+\sum_{i\ge2} a_i^{(n)}\xi_{n,i},
\qquad
Y_n:=b_1^{(n)}\eta+\sum_{i\ge2} b_i^{(n)}\xi_{n,i},
\]
we have
\[
\norm{\abs{X_n}-\abs{Y_n}}_{L^2}\longrightarrow 0.
\]
\end{enumerate}
\end{proposition}

\begin{proof}
Assume that the conclusion of Theorem~\ref{thm:main} is false.  Then there is
no constant $C$ for which the stable phase retrieval estimate holds uniformly
on the closed span of the coordinates.  By the orthogonal reduction
(Proposition~\ref{prop:orth-reduction}), this means that we may find, for each
$n\in\N$, an orthogonal pair of unit vectors in the span whose modulus separation is
at most $1/n$.  Concretely, there exist functions $X_n,Y_n$ in
$\overline{\Span}(\eta,\xi_i:i\ge2)$ such that
\[
\norm{X_n}_{L^2}=\norm{Y_n}_{L^2}=1,
\qquad
\E[X_nY_n]=0,
\qquad
\norm{|X_n|-|Y_n|}_{L^2}\le \frac1n.
\]

\noindent Since the family $(\eta,\xi_i)_{i\ge2}$ is orthonormal in $L^2$, we may write
\[
X_n=a_1^{(n)}\eta+\sum_{i\ge2}a_i^{(n)}\xi_i,
\qquad
Y_n=b_1^{(n)}\eta+\sum_{i\ge2}b_i^{(n)}\xi_i,
\]
with coefficient vectors $a^{(n)},b^{(n)}\in\ell^2(\N)$.  Parseval's identity
then yields
\[
\norm{a^{(n)}}_2=\norm{X_n}_{L^2}=1,
\qquad
\norm{b^{(n)}}_2=\norm{Y_n}_{L^2}=1,
\]
and, since the two functions are orthogonal in $L^2$,
\[
\ip{a^{(n)}}{b^{(n)}}=\E[X_nY_n]=0.
\]
Finally, to prepare for later row-wise permutations of the good coordinates (the ones with $L^1$ norm bounded away from one), we
rename the good variables used in the $n$-th row by writing them as
$(\xi_{n,i})_{i\ge2}$.  This does not alter their laws or any of the stated
hypotheses, but it is convenient for the compactness scheme.  The displayed
modulus convergence in (ii) follows from our chosen violating sequence.  This completes the proof of the
proposition.
\end{proof}
Our next lemma follows a typical strategy in concentration compactness arguments where one splits a potential sequence of counterexamples into 
pieces with extremal behavior. The terminology in the lemma is consistent with the previous proof. More precisely, the \emph{good coordinates} are the ones with $L^1$ norm bounded away from one and the \emph{exceptional coordinate} is the variable that does not have this property.

\begin{lemma}[Profile decomposition]\label{lem:profile}
After passing to a subsequence and, row by row, applying the same permutation
to the good coordinates and to their coefficients, one may find vectors
$a,b\in\ell^2(\N)$ and integers $m_n\to\infty$ such that, with
\[
a_{\mathrm{head}}^{(n)}:=a^{(n)}\mathbf 1_{\{1,\dots,m_n\}},
\qquad
a_{\mathrm{tail}}^{(n)}:=a^{(n)}\mathbf 1_{\{m_n+1,m_n+2,\dots\}},
\]
and likewise for $b^{(n)}$, the following hold:
\begin{enumerate}[label=\rm(\arabic*)]
\item $a_{\mathrm{head}}^{(n)}\to a$ and $b_{\mathrm{head}}^{(n)}\to b$
strongly in $\ell^2$.
\item We have the convergence
\[
\max_{i>m_n}\max(\abs{a_i^{(n)}},\abs{b_i^{(n)}})\to0.
\]
\item We have the convergence
\[
\norm{a_{\mathrm{tail}}^{(n)}}_2^2\to1-\norm{a}_2^2,
\qquad
\norm{b_{\mathrm{tail}}^{(n)}}_2^2\to1-\norm{b}_2^2,
\]
and
\[
\ip{a_{\mathrm{tail}}^{(n)}}{b_{\mathrm{tail}}^{(n)}}
\to -\ip{a}{b}.
\]
\end{enumerate}
\end{lemma}

\begin{proof}
For each row $n$ and each good index $i\ge2$, define
\[
s_i^{(n)}:=\max\bigl(\abs{a_i^{(n)}},\abs{b_i^{(n)}}\bigr).
\]
Permuting the good coordinates within the $n$-th row, we may assume that the
sequence $(s_i^{(n)})_{i\ge2}$ is nonincreasing in $i$.  The same permutation
is applied simultaneously to the variables $\xi_{n,i}$ and to both coefficient
vectors, so our hypothesis is not changed.
This rearrangement gives the following uniform pointwise bound on the reordered
coefficients:
\[
\sum_{i\ge2}(s_i^{(n)})^2
\le
\sum_{i\ge2}\bigl((a_i^{(n)})^2+(b_i^{(n)})^2\bigr)
\le 2.
\]
Since $(s_i^{(n)})_{i\ge2}$ is nonincreasing, for every $i\ge2$ we have
\[
(i-1)(s_i^{(n)})^2
\le \sum_{j=2}^{i}(s_j^{(n)})^2
\le 2,
\]
and therefore
\[
s_i^{(n)}\le \sqrt{\frac{2}{i-1}}.
\]
In particular, for each fixed coordinate $i$, the sequences
$\bigl(a_i^{(n)}\bigr)_{n\ge1}$ and $\bigl(b_i^{(n)}\bigr)_{n\ge1}$ are
bounded.

We now extract coordinatewise limits.  First, we pass to a subsequence so that
$a_1^{(n)}$ and $b_1^{(n)}$ converge.  Next, we pass to a further subsequence so
that $a_2^{(n)}$ and $b_2^{(n)}$ converge. We then continue inductively.  A
standard diagonal argument yields a single subsequence, not relabeled, for
which every fixed coordinate converges:
\[
a_i^{(n)}\to a_i,
\qquad
b_i^{(n)}\to b_i
\qquad(i\ge1).
\]
Set
\[
a:=(a_i)_{i\ge1},
\qquad
b:=(b_i)_{i\ge1}.
\]
We claim that $a,b\in\ell^2(\N)$.  To see this, fix $M\in\N$.  Since the first $M$
coordinates converge, we have
\[
\sum_{i=1}^{M} a_i^2
=\lim_{n\to\infty}\sum_{i=1}^{M}(a_i^{(n)})^2
\le \liminf_{n\to\infty}\sum_{i\ge1}(a_i^{(n)})^2
=1.
\]
Letting $M\to\infty$ gives
$\sum_i a_i^2\le1$.  The same argument yields $\sum_i b_i^2\le1$.
Knowing now that $a,b\in\ell^2$, we may choose an increasing sequence of integers
$(m_k)_{k\ge1}$ with $m_k\to\infty$ such that
\[
\sum_{i>m_k} a_i^2+\sum_{i>m_k} b_i^2\le 2^{-k}
\qquad(k\ge1).
\]
For each fixed $k$, coordinatewise convergence implies that
\[
\sum_{i=1}^{m_k}\abs{a_i^{(n)}-a_i}^2
+
\sum_{i=1}^{m_k}\abs{b_i^{(n)}-b_i}^2
\longrightarrow0
\qquad(n\to\infty).
\]
Passing to another subsequence if necessary, we may assume that for each $n$ we have
\[
\sum_{i=1}^{m_n}\abs{a_i^{(n)}-a_i}^2
+
\sum_{i=1}^{m_n}\abs{b_i^{(n)}-b_i}^2
\le 2^{-n}.
\]
We now verify the three conclusions.  First, we note that
\[
\norm{a_{\mathrm{head}}^{(n)}-a}_2^2
\le
\sum_{i=1}^{m_n}\abs{a_i^{(n)}-a_i}^2
+
\sum_{i>m_n}a_i^2
\le 2^{-n}+2^{-n},
\]
so $a_{\mathrm{head}}^{(n)}\to a$ in $\ell^2$.  The same estimate gives
$b_{\mathrm{head}}^{(n)}\to b$.
Second, if $i>m_n$, then by monotonicity we have
\[
\max\bigl(\abs{a_i^{(n)}},\abs{b_i^{(n)}}\bigr)
=s_i^{(n)}
\le s_{m_n+1}^{(n)}
\le \sqrt{\frac{2}{m_n}}.
\]
Since $m_n\to\infty$, this proves
\[
\max_{i>m_n}\max\bigl(\abs{a_i^{(n)}},\abs{b_i^{(n)}}\bigr)\to0.
\]
Finally, because each full coefficient vector has $\ell^2$ norm one,
\[
\norm{a_{\mathrm{tail}}^{(n)}}_2^2
=1-\norm{a_{\mathrm{head}}^{(n)}}_2^2
\longrightarrow 1-\norm{a}_2^2,
\]
and similarly
\[
\norm{b_{\mathrm{tail}}^{(n)}}_2^2
\longrightarrow 1-\norm{b}_2^2.
\]
For the inner products, we use the orthogonality of the full vectors to split
\[
0
=\ip{a^{(n)}}{b^{(n)}}
=\ip{a_{\mathrm{head}}^{(n)}}{b_{\mathrm{head}}^{(n)}}
+\ip{a_{\mathrm{tail}}^{(n)}}{b_{\mathrm{tail}}^{(n)}}.
\]
Since the head component converges strongly, we have
\[
\ip{a_{\mathrm{head}}^{(n)}}{b_{\mathrm{head}}^{(n)}}
\to \ip{a}{b},
\]
and therefore
\[
\ip{a_{\mathrm{tail}}^{(n)}}{b_{\mathrm{tail}}^{(n)}}
\to -\ip{a}{b}.
\]
This proves the lemma.
\end{proof}

\subsection{Identification of the head component} Our next step is to identify the  behavior of the head profiles constructed in the previous step when $n$ is sufficiently large. To this end, we start with an analogue of Lemma \ref{lem:probes} for the exceptional coordinate $\eta$.

\begin{lemma}[Probes for the exceptional coordinate]\label{prop:exceptional-probes}
Let the notation be as above. Then there exists a bounded measurable function $S_\eta$ such that
\[
\E[S_\eta]=0,
\qquad
\E[\eta S_\eta]=1.
\]
Moreover, if $\eta^2\not\equiv1$ almost surely, then there exists a bounded measurable
function $T_\eta$ such that
\[
\E[T_\eta]=0,
\qquad
\E[(\eta^2-1)T_\eta]=1.
\]
\end{lemma}

\begin{proof}
Since $\norm{\eta}_{L^2}=1$, we have $\E[\abs{\eta}]>0$.  Applying the linear probe construction from
Lemma~\ref{lem:probes} with $A=\E[\abs{\eta}]$ produces a bounded measurable
function $S_\eta$ satisfying
\[
\E[S_\eta]=0,
\qquad
\E[\eta S_\eta]=1.
\]
Assume now that $\eta^2\not\equiv1$ almost surely, so that
\[
\E[\abs{\eta^2-1}]>0.
\]
Applying again the linear probe construction, now to the random variable
$\eta^2-1$, we obtain a bounded measurable function $T_\eta$ such that
\[
\E[T_\eta]=0,
\qquad
\E[(\eta^2-1)T_\eta]=1.
\]
This is exactly the desired conclusion.
\end{proof}

\begin{lemma}[Identification of the head profile]\label{lem:head-sign}
Let the notation be as above. Then either there exists $\tau\in\{\pm1\}$ such that
\[
a=\tau b
\qquad\text{and}\qquad
\norm{a_{\mathrm{head}}^{(n)}-\tau b_{\mathrm{head}}^{(n)}}_2\to0,
\]
or else $\eta$ is a Rademacher variable and
\[
a_i=b_i=0
\qquad(i\ge2).
\]
\end{lemma}

\begin{proof}
 Since
\[
\norm{X_n}_{L^2}=\norm{Y_n}_{L^2}=1,
\]
we have
\[
\norm{|X_n|+|Y_n|}_{L^2}\le 2.
\]
Therefore, by H\"older and the fact that our random variables are real-valued, we have
\[
\norm{X_n^2-Y_n^2}_{L^1}\longrightarrow0.
\]
We now analyze several cases in detail.
 \vspace{0.5em}
 
\noindent\emph{Case 1: diagonal terms.}
Fix $i\ge2$ and let $T_{n,i}$ be the quadratic probe associated with the
coordinate $\xi_{n,i}$.  The probes are uniformly bounded in $L^\infty$, so the above
$L^1$ convergence implies that
\[
\E[(X_n^2-Y_n^2)T_{n,i}]\longrightarrow0.
\]
We now compute the left-hand side explicitly.  Expanding $X_n^2$ and integrating
against $T_{n,i}$, every term involving a coordinate different from
$\xi_{n,i}$ vanishes because $T_{n,i}$ depends only on $\xi_{n,i}$ and is
centered.  Mixed terms involving $\xi_{n,i}$ and another coordinate also
vanish by independence and centering.  The only surviving contribution is the
term $(a_i^{(n)})^2\xi_{n,i}^2$.  Therefore, we have
\[
\E[X_n^2T_{n,i}]
=(a_i^{(n)})^2\E[\xi_{n,i}^2T_{n,i}]
=(a_i^{(n)})^2,
\]
because
\[
\E[\xi_{n,i}^2T_{n,i}]
=\E[(\xi_{n,i}^2-1)T_{n,i}]+\E[T_{n,i}]
=1.
\]
Similarly,
\[
\E[Y_n^2T_{n,i}]=(b_i^{(n)})^2.
\]
Therefore,
\[
(a_i^{(n)})^2-(b_i^{(n)})^2
=\E[(X_n^2-Y_n^2)T_{n,i}]
\longrightarrow0.
\]
Passing to the limit in $n$ gives
\[
a_i^2=b_i^2
\qquad(i\ge2).
\]
\vspace{0.1em}

\noindent\emph{Case 2: off-diagonal terms.}
Fix distinct indices $i,j\ge2$ and let $S_{n,i},S_{n,j}$ be the corresponding
linear probes.  Again, boundedness and $L^1$ convergence imply that
\[
\E[(X_n^2-Y_n^2)S_{n,i}S_{n,j}]\longrightarrow0.
\]
In the expansion of $X_n^2$, all terms vanish after integration against
$S_{n,i}S_{n,j}$ except the mixed term
$2a_i^{(n)}a_j^{(n)}\xi_{n,i}\xi_{n,j}$.  Indeed, any term missing one of the
coordinates $\xi_{n,i}$ or $\xi_{n,j}$ is killed by the zero mean of the
corresponding probe, and terms involving other coordinates factor through zero
means by independence.  Hence,
\[
\E[X_n^2S_{n,i}S_{n,j}]
=2a_i^{(n)}a_j^{(n)}
\E[\xi_{n,i}S_{n,i}]\,\E[\xi_{n,j}S_{n,j}]
=2a_i^{(n)}a_j^{(n)}.
\]
Similarly,
\[
\E[Y_n^2S_{n,i}S_{n,j}]
=2b_i^{(n)}b_j^{(n)}.
\]
Thus, we have
\[
2(a_i^{(n)}a_j^{(n)}-b_i^{(n)}b_j^{(n)})\longrightarrow0,
\]
and passing to the limit gives
\[
a_ia_j=b_ib_j
\qquad(i\ne j,\ i,j\ge2).
\]
\vspace{0.1em}

\noindent\emph{Case 3: the exceptional term.}
Fix $j\ge2$.  Let $S_\eta$ be the bounded probe from
Lemma~\ref{prop:exceptional-probes}.  By the same reasoning, we have
\[
\E[(X_n^2-Y_n^2)S_\eta S_{n,j}]\longrightarrow0.
\]
Once more, only one mixed term survives, namely
$2a_1^{(n)}a_j^{(n)}\eta\xi_{n,j}$ for $X_n^2$ and the corresponding term for
$Y_n^2$.  Therefore,
\[
\E[X_n^2S_\eta S_{n,j}]
=2a_1^{(n)}a_j^{(n)}\E[\eta S_\eta]\E[\xi_{n,j}S_{n,j}]
=2a_1^{(n)}a_j^{(n)},
\]
and analogously
\[
\E[Y_n^2S_\eta S_{n,j}]
=2b_1^{(n)}b_j^{(n)}.
\]
Hence,
\[
a_1a_j=b_1b_j
\qquad(j\ge2).
\]
 \vspace{0.1em}
 
\noindent\emph{Step 4: deduction of the sign.}
Assume first that there exists some index $i_0\ge2$ such that $a_{i_0}\ne0$.
From the identity of the diagonal terms we know that
$a_{i_0}^2=b_{i_0}^2$, so $b_{i_0}=\tau a_{i_0}$ for some
$\tau\in\{\pm1\}$.  For every $j\ge2$, the off-diagonal identity then gives
\[
a_{i_0}a_j=b_{i_0}b_j=\tau a_{i_0}b_j.
\]
Since $a_{i_0}\ne0$, we conclude that $a_j=\tau b_j$ for every $j\ge2$.
Applying the identity in Case 3 with $j=i_0$, we obtain
\[
a_1a_{i_0}=b_1b_{i_0}=\tau a_{i_0}b_1,
\]
and again $a_{i_0}\ne0$ implies $a_1=\tau b_1$.  Hence, $a=\tau b$.

Assume next that
\[
a_i=0
\qquad\text{for every }i\ge2.
\]
In this case, the diagonal identities show that we also have $b_i=0$ for every $i\ge2$.
If $\eta^2\not\equiv1$ almost surely, let $T_\eta$ be the quadratic probe from
Lemma~\ref{prop:exceptional-probes}.  Repeating the diagonal argument for
the exceptional coordinate gives
\[
(a_1^{(n)})^2-(b_1^{(n)})^2
=\E[(X_n^2-Y_n^2)T_\eta]
\longrightarrow0,
\]
so $a_1^2=b_1^2$ and we again conclude that $a=\tau b$ for some sign $\tau$.

The only remaining possibility is that $\eta^2\equiv1$ almost surely.  Since
$\E[\eta]=0$ and $\norm{\eta}_{L^2}=1$, this means exactly that $\eta$ is a
Rademacher variable.  In this case, we are in the second alternative stated in
the lemma.

Finally, if the first alternative holds, then the convergence of the head
profiles follows from Lemma~\ref{lem:profile}, since
\[
\norm{a_{\mathrm{head}}^{(n)}-\tau b_{\mathrm{head}}^{(n)}}_2
\le
\norm{a_{\mathrm{head}}^{(n)}-a}_2
+
\norm{b_{\mathrm{head}}^{(n)}-b}_2
\longrightarrow0.
\]
This finishes the proof.
\end{proof}

\subsection{Properties of the limit} We next identify the properties that the limiting distributions stemming from sequences failing SPR must satisfy. 

\begin{lemma}\label{lem:limit-law}
Let the notation be as above. After passing to a subsequence, the quadruple
\[
\bigl(X_n^{\mathrm{head}},Y_n^{\mathrm{head}},
X_n^{\mathrm{tail}},Y_n^{\mathrm{tail}}\bigr)
\]
converges in distribution to a random vector
\[
(H_X,H_Y,R_X,R_Y)\in\R^4
\]
such that
\begin{enumerate}[label=\rm(\arabic*)]
\item $(H_X,H_Y)$ is independent of $(R_X,R_Y)$;
\item the law of $(R_X,R_Y)$ is infinitely divisible;
\item either $H_X=\tau H_Y$ almost surely for some $\tau\in\{\pm1\}$, or
else $\eta$ is Rademacher and
$H_X=a_1\eta$, $H_Y=b_1\eta$ almost surely;
\item
\[
\abs{H_X+R_X}=\abs{H_Y+R_Y}
\qquad\text{almost surely.}
\]
\end{enumerate}
\end{lemma}

\begin{proof}
The head variables converge in $L^2$ by the strong convergence of the head
coefficients and the orthonormality of each row.  The tail coefficient arrays
are infinitesimal by Lemma~\ref{lem:profile}, so every weak limit of the
tail pair is infinitely divisible by Lemma \ref{lem:diffuse-id-limit}.  Independence of the head and the tail persists
because they are built from disjoint independent coordinates.  Finally,
$\norm{\abs{X_n}-\abs{Y_n}}_{L^2}\to0$ implies
$\abs{X_n}-\abs{Y_n}\to0$ in probability, so passing to the joint weak limit
gives the desired identity in the limit.  The description of the head follows
from Lemma~\ref{lem:head-sign}.
\end{proof}
We are now in a position to begin understanding the weak limits obtained above, with the objective of obtaining a contradiction. The key geometric observation
is contained in the following lemma.
\begin{lemma}\label{lem:id-line}
Let $\mu$ be an infinitely divisible probability measure on $\R^2$ whose support satisfies
\[
\supp\mu\subset\{(u,v)\in\R^2:\abs{u}=\abs{v}\} \eqqcolon \Gamma.
\]
Then there exists $\tau\in\{\pm1\}$ such that
\[
\supp\mu\subset \{(u,\tau u):u\in\R\} \eqqcolon L_\tau.
\]
\end{lemma}

\begin{proof}
The set $\Gamma$ is exactly the union of two lines $L_+\cup L_-$.  Since $\mu$ is
infinitely divisible, it is in particular $2$-divisible, meaning that there exists a
probability measure $\nu$ such that $\mu=\nu*\nu$.  Let $A=\supp\nu$.  Then
\[
\supp\mu=\overline{A+A}\subset L_+\cup L_-.
\]
If $x,y\in A$ are distinct, then $2x,x+y,2y\in A+A$.  If $2x$ and $2y$ belong
to the same line $L_\tau$, then so does $x+y$, forcing $x,y\in L_\tau$.  If instead
they belong to different lines, then the midpoint $x+y$ is the origin, so
$y=-x$ and again $x,y$ lie in one of the lines.  Thus, every two points of $A$
lie on a common  line.  Since $L_+\cap L_-=\{0\}$, this implies
that $A$ is contained in one of the two lines, and so is $\supp\mu$.
\end{proof}

\subsection{Conclusion of the compactness proof}

\begin{proof}[Proof of Theorem~\ref{thm:main}]
Assume that Theorem~\ref{thm:main} fails, let the counterexample array be as
in Proposition~\ref{prop:counterexample} and apply
Lemmas~\ref{lem:profile}, \ref{lem:head-sign}, and \ref{lem:limit-law}.

If the first alternative of Lemma~\ref{lem:head-sign} holds, fix
$\tau\in\{\pm1\}$ such that $H_Y=\tau H_X$ almost surely.  The  modulus
identity becomes
\[
\abs{H_X+R_X}=\abs{\tau H_X+R_Y}
\qquad\text{almost surely.}
\]
If $H_X$ is not almost surely zero, then conditioning on the tail and varying
the head over two support points of $H_X$ shows that
$R_Y-\tau R_X=0$ almost surely.  If $H_X\equiv0$, then
$(R_X,R_Y)$ itself satisfies $\abs{R_X}=\abs{R_Y}$ almost surely, so
Lemma~\ref{lem:id-line} yields
$R_Y=\tau_0 R_X$ almost surely for some $\tau_0\in\{\pm1\}$.

In either case, there exists a sign $\sigma\in\{\pm1\}$ for which the random variables
\[
Z_n:=\sum_{i>m_n}(b_i^{(n)}-\sigma a_i^{(n)})\xi_{n,i}
\]
converge to $0$ in distribution.  Since the coefficients are diffuse,
Lemma~\ref{lem:tail-rigidity} gives
\[
\norm{b_{\mathrm{tail}}^{(n)}-\sigma a_{\mathrm{tail}}^{(n)}}_2\to0.
\]
Together with Lemma~\ref{lem:head-sign} this implies that
$\norm{b^{(n)}-\sigma a^{(n)}}_2\to0$, contradicting
\[
\norm{b^{(n)}-\sigma a^{(n)}}_2^2
=\norm{a^{(n)}}_2^2+\norm{b^{(n)}}_2^2
-2\sigma\ip{a^{(n)}}{b^{(n)}}=2.
\]
If the second alternative of Lemma~\ref{lem:head-sign} holds, then
$\eta$ is Rademacher and the head limits are multiples of $\eta$ alone.  The
same argument therefore applies, since the good coordinates disappear from
the head and the tail limit still lies on $\Gamma$.
\end{proof}

\subsection{An alternative argument}\label{ssec:gaussian-alt}
The previous argument used only the support structure of infinitely divisible
limits.  The discussion below sketches a second route which may be used to complete the compactness proof, showing that in the regime of a degenerate head component, the
diffuse tail automatically generates a nontrivial Gaussian component.

We start with a result that ensures that, in spite of the failure of the Lindeberg condition for the Central Limit Theorem, one may still ensure the existence of a nondegenerate Gaussian component of the limiting distribution. 
\begin{proposition}
\label{prop:gaussian-part}
Let
\[
V_{n,i}:=(a_i^{(n)}\xi_{n,i},\,b_i^{(n)}\xi_{n,i})\in\R^2
\qquad(i>m_n),
\]
where the coordinates are independent, centered, $L^2$-normalized and satisfy
$\norm{\xi_{n,i}}_{L^1}\ge \delta$.  Set
\[
\alpha_{n,i}:=\bigl((a_i^{(n)})^2+(b_i^{(n)})^2\bigr)^{1/2}.
\]
Assume that
\[
\max_{i>m_n}\alpha_{n,i}\to0,
\qquad
\sum_{i>m_n}\alpha_{n,i}^2\to \sigma^2>0,
\]
and that the row sums $\sum_{i>m_n}V_{n,i}$ converge in distribution to an
infinitely divisible law $\mu$ on $\R^2$.  Then the covariance matrix of the Gaussian component of $\mu$ is not identically zero.
\end{proposition}

\begin{proof}[Sketch of proof]
Set $\beta_n=\max_{i>m_n}\alpha_{n,i}$ and $r_n=\beta_n^{1/2}$.  For each
$i>m_n$ with $\alpha_{n,i}\ne0$, let $K_{n,i}=r_n/\alpha_{n,i}\to\infty$
uniformly and define the centered truncation
\[
\widetilde\xi_{n,i}
:=
\xi_{n,i}\mathbf 1_{\{\abs{\xi_{n,i}}\le K_{n,i}\}}
-\E[\xi_{n,i}\mathbf 1_{\{\abs{\xi_{n,i}}\le K_{n,i}\}}].
\]
Since $\E[\abs{\xi_{n,i}}]\ge\delta$ and $\E[\xi_{n,i}^2]=1$, a simple estimate shows that there exists $c_0=c_0(\delta)>0$ such that
\[
\E[\widetilde\xi_{n,i}^2]\ge c_0
\]
for all large $n$ and all $i>m_n$.  Define
$\widetilde V_{n,i}=(a_i^{(n)}\widetilde\xi_{n,i},\,b_i^{(n)}\widetilde\xi_{n,i})$.
Then $\abs{\widetilde V_{n,i}}\le 2r_n$, so the truncated array is bounded,
infinitesimal and we have
\[
\sum_{i>m_n}\E[\abs{\widetilde V_{n,i}}^2]
\ge c_0\sum_{i>m_n}\alpha_{n,i}^2
\longrightarrow c_0\sigma^2>0.
\]
For infinitesimal triangular arrays, the covariance matrices of shrinking
centered truncations converge to the Gaussian covariance of the limit law (see
\cite[Chapter~IV]{gnedenko-kolmogorov}).  Since the traces of the truncation
covariances stay bounded away from zero, the Gaussian covariance of $\mu$ is
nonzero.
\end{proof}

Using Proposition~\ref{prop:gaussian-part}, we may present an alternative route to the conclusion of the compactness proof of Theorem~\ref{thm:main}. Assume that we are in the degenerate head regime $H_X\equiv H_Y\equiv0$ (as the other one is essentially trivial) so the
tail component $\mu$, the law of $(R_X,R_Y)$, is supported on $\Gamma=\{(u,v):\abs{u}=\abs{v}\}$.
Write $\mu=\gamma*\nu$, where $\gamma$ is the Gaussian part and $\nu$ the
independent remainder term.  By Proposition~\ref{prop:gaussian-part}, $\gamma$ is
nontrivial.  Its support is a nonzero linear subspace $E\subset\R^2$.
Since $\supp\mu=\supp\gamma+\supp\nu\subset\Gamma=L_+\cup L_-$, the
connected set $E$ must itself lie in one of the lines $L_\tau$.  Fix such a
$\tau$.  For any $z\in\supp\nu$, the translated line $E+z$ is contained in
$\Gamma$.  Since a translate of $L_\tau$ not equal to $L_\tau$ intersects
$L_+\cup L_-$ in at most two points, one must have $z\in L_\tau$.  It follows that
$\supp\nu\subset L_\tau$ and therefore $\supp\mu\subset L_\tau$.  Thus,
$R_Y=\tau R_X$ almost surely, and this recovers the final contradiction in the first proof of Theorem \ref{thm:main}. 

\section{Second Proof: Anticoncentration estimates}\label{sec:quant-proof}

This section gives a second proof of Theorem~\ref{thm:main}.  The argument is
quantitative in spirit.  Instead of passing to a limit law, we split the
coefficients into a head and a tail and show directly that one of two explicit
mechanisms forces a non-trivial separation of the modulus of the random variables.

 By Proposition~\ref{prop:orth-reduction}, it is enough to prove a uniform
lower bound for arbitrary orthonormal vectors $X, Y$ in the closed span of the random variables. By a standard density argument, we may consider the case with only finitely many (say, $n$) random variables, as long as the constants do not depend on $n$.

\subsection{The dichotomy setup}
Let the notation be as above. We start with the basic setup for our second proof, which is based on a simple dichotomy. Let
\[
\eps:=\E[\abs{X^2-Y^2}],
\qquad
\Delta:=\norm{|X|-|Y|}_{L^2},
\]
and notice that (under the above normalization assumptions)
\[
\eps\le 2\Delta.
\]
Fix
\[
\kappa_{\delta,\eta}
:=
\min\left(
\frac{\delta}{4\sqrt2},
\frac{\E[\abs{\eta-\widetilde\eta}]}{2\sqrt2}
\right)>0,
\]
where $\widetilde\eta$ is an independent copy of $\eta$.  Next, set
\[
\theta:=\min\bigl(\delta^2/8,\delta\kappa_{\delta,\eta}\bigr),
\qquad
\lambda:=\min\bigl(\delta^3/2^{20},\delta\theta/2^{17}\bigr),
\qquad
s_0:=\frac{\theta^2}{2^{14}}.
\]
Let $S_\eta$ be the bounded linear probe from
Lemma~\ref{prop:exceptional-probes}. Choose a
constant $K_{\delta,\eta}$ so that
\[
K_{\delta,\eta}\ge \max\bigg\{\frac{4}{\delta^2},\frac{2}{\delta
}\norm{S_\eta}_{L^\infty}\bigg\} \ge 1.
\]
 In the case $\eta^2\not\equiv1$, let
$T_\eta$ be the bounded quadratic probe and increase $K_{\delta,\eta}$ (if necessary) to also satisfy
\[
K_{\delta,\eta}\ge \frac{2}{\delta} \norm{T_\eta}_{L^\infty}.
\]
Define the head coordinates $H$ by
\[
H(a,b):=\{1\}\cup\{i\ge2:\max(|a_i|,|b_i|)>\lambda\}.
\]
Note that the cardinality of $H$ is bounded by
\[
\#H\le 3+2\lambda^{-2} \eqqcolon d^{\ast}(\delta,\eta),
\]
and that every tail coordinate $i\notin H$ satisfies
\[
\max(|a_i|,|b_i|)\le \lambda.
\]
The first step is to separate the proof into two mutually exclusive regimes.
We record our simple case split in the following lemma, whose proof we omit. 
\begin{lemma}[Tail dichotomy]\label{lem:dichotomy}
Exactly one of the following alternatives holds.
\begin{enumerate}[label=\rm(\roman*)]
\item \textit{Diffuse case:} We have
\[
\norm{(a-b)\mathbf 1_{H^c}}_2^2\ge16s_0
\quad\text{and}\quad
\norm{(a+b)\mathbf 1_{H^c}}_2^2\ge16s_0.
\]
\item \textit{Concentrated case:} There exists $\tau\in\{\pm1\}$ such that
\[
\norm{(a-\tau b)\mathbf 1_{H^c}}_2<\frac{\theta}{32}.
\]
\end{enumerate}
\end{lemma}

\subsection{The concentrated case}
In the concentrated case, the tail of
$a-\tau b$ is small for a suitable sign $\tau$, so any substantial discrepancy
between the two vectors must be visible from the head.  The next proposition
extracts information about the
head coefficients.

\begin{proposition}
\label{prop:head-entry}
The following estimates hold.
\begin{enumerate}[label=\rm(\arabic*)]
\item For any indices $i,j\in H\cap\{2,3,\dots\}$,
\[
|a_i a_j-b_i b_j|\le K_{\delta,\eta}\,\eps.
\]
\item For any index $j\in H\cap\{2,3,\dots\}$,
\[
|a_1a_j-b_1b_j|\le K_{\delta,\eta}\,\eps.
\]
\item If $\eta^2\not\equiv1$, then
\[
|a_1^2-b_1^2|\le K_{\delta,\eta}\,\eps.
\]
\end{enumerate}
\end{proposition}

\begin{proof}
For $i=j\ge2$, we use the quadratic probe $T_i$ for $\xi_i$ to estimate
\[
|a_i^2-b_i^2|
=\bigl|\E[(X^2-Y^2)T_i]\bigr|
\le \norm{T_i}_{L^\infty}\eps
\le K_{\delta,\eta}\eps.
\]
For distinct $i,j\ge2$, we use the linear probes $S_i,S_j$ to estimate
\[
2|a_i a_j-b_i b_j|
=\bigl|\E[(X^2-Y^2)S_iS_j]\bigr|
\le \norm{S_i}_{L^\infty}\norm{S_j}_{L^\infty}\eps
\le K_{\delta,\eta}\eps.
\]
This proves \rm(1).

For \rm(2), we use the product $S_\eta S_j$.  Exactly as in the proof of
Lemma~\ref{lem:head-sign}, all the terms vanish except one, so
\[
2|a_1a_j-b_1b_j|
=\bigl|\E[(X^2-Y^2)S_\eta S_j]\bigr|
\le \norm{S_\eta}_{L^\infty}\norm{S_j}_{L^\infty}\eps
\le K_{\delta,\eta}\eps.
\]
Finally, if $\eta^2\not\equiv1$, we use $T_\eta$ to estimate
\[
|a_1^2-b_1^2|
=\bigl|\E[(X^2-Y^2)T_\eta]\bigr|
\le \norm{T_\eta}_{L^\infty}\eps
\le K_{\delta,\eta}\eps.
\]
\end{proof}

The following lemma conveniently packages the above quantities into a single matrix.
\begin{lemma}\label{lem:rank-one}
For $x,y\in\R^m$, let $A(x,y)_{ij}=x_ix_j-y_iy_j$ be the difference matrix between $x$ and $y$. Then
\[
\norm{x-y}_2^2\,\norm{x+y}_2^2\le 2\norm{A(x,y)}_F^2.
\]
\end{lemma}

\begin{proof}
Writing $s=\norm{x}_2^2$, $t=\norm{y}_2^2$ and $c=\ip{x}{y}$, the left-hand
side equals $(s+t)^2-4c^2$, while the right-hand side is
$2(s^2+t^2-2c^2)$.  Thus, the right-hand side has an extra factor of $(s-t)^2$, which is nonnegative.
\end{proof}

\begin{proposition}[Concentrated case]\label{prop:concentrated}
If there exists $\tau\in\{\pm 1\}$ such that
\[
\norm{(a-\tau b)\mathbf 1_{H^c}}_2\le \frac{\theta}{32},
\]
then there exists a constant $c_{\mathrm{line}}(\delta,\eta)>0$ such that
\[
\eps\ge c_{\mathrm{line}}(\delta,\eta).
\]
\end{proposition}

\begin{proof}
We distinguish two cases.
\vspace{0.5em}

\noindent\emph{Case 1: $\eta$ is not Rademacher or $H\cap\{2,3,\dots\}$ is
nonempty.}
In this case, the head coefficients are linked by
Proposition~\ref{prop:head-entry}. 
In particular, by direct calculation and an argument similar to Lemma~\ref{lem:head-sign} in the Rademacher case, 
there exists a sign, which must be the same $\tau$, such that the
difference matrix of $a\mathbf 1_H$ and $\tau b\mathbf 1_H$ has all
entries bounded by $K_{\delta,\eta}\eps$. Thus, applying Lemma~\ref{lem:rank-one} with
\[
x:=a\mathbf 1_H,
\qquad
y:=\tau b\mathbf 1_H,
\]
we obtain 
\[
\norm{x-y}_2^2\,\norm{x+y}_2^2
\le 2\#H^2 K_{\delta,\eta}^2\eps^2.
\]
Since $(a-\tau b)\mathbf{1}_{H^c}$ has norm at most $\theta/32$ and
$\norm{a-\tau b}_2^2=2$ by orthogonality between $a$ and $b$, we have
\[
\norm{x-y}_2^2
=2-\norm{(a-\tau b)\mathbf 1_{H^c}}_2^2
\ge 1.
\]
	Hence,
	\[
	\norm{x+y}_2\le \sqrt{2}(\#H) K_{\delta,\eta}\eps
	\le \sqrt{2}d^*(\delta,\eta)K_{\delta,\eta}\eps.
	\]
	Set
	\[
	\alpha:=\norm{(a+\tau b)\mathbf 1_H}_2,
	\qquad
	\beta:=\norm{(a-\tau b)\mathbf 1_{H^c}}_2,
	\]
	and note that
	\[
	\alpha\le \sqrt2\,d^*(\delta,\eta)K_{\delta,\eta}\eps,
	\qquad
	\beta\le \theta/32.
	\]
	We now describe the gluing step.  Write
	\[
	p:=(a-\tau b)\mathbf 1_H,
	\qquad
	q:=(a-\tau b)\mathbf 1_{H^c},
	\]
	\[
	r:=(a+\tau b)\mathbf 1_H,
	\qquad
	s:=(a+\tau b)\mathbf 1_{H^c},
	\]
	and let
	\[
	P:=p_1\eta+\sum_{i\in H\cap\{2,3,\dots\}} p_i\xi_i,
	\qquad
	Q:=\sum_{i\notin H} q_i\xi_i,
	\]
	\[
	R:=r_1\eta+\sum_{i\in H\cap\{2,3,\dots\}} r_i\xi_i,
	\qquad
	S:=\sum_{i\notin H} s_i\xi_i.
	\]
	We may decompose
	\[
	X-\tau Y=P+Q,
	\qquad
	X+\tau Y=R+S,
	\]
	so that
	\[
	X^2-Y^2
	=(X-\tau Y)(X+\tau Y)
	=(P+Q)(R+S).
	\]
	Since the supports of $p,r$ and $q,s$ are disjoint, the pair $(P,R)$ is
	independent of the pair $(Q,S)$. In particular, $P$ and $S$ are independent.
	In view of the inequality
	\[
	\abs{PS}\le \abs{X^2-Y^2}+\abs{PR}+\abs{QS}+\abs{QR},
	\]
	 we deduce that
	\begin{equation}
    \label{eq:aux_conc_branch}
	\E[\abs{P}]\,\E[\abs{S}]
	=\E[\abs{PS}]
	\le
	\eps+\E[\abs{PR}]+\E[\abs{QS}]+\E[\abs{QR}].
	\end{equation}
	By Cauchy--Schwarz, we also have
	\[
	\E[\abs{PR}]\le \norm{P}_{L^2}\norm{R}_{L^2},
	\qquad
	\E[\abs{QS}]\le \norm{Q}_{L^2}\norm{S}_{L^2},
	\qquad
	\E[\abs{QR}]\le \norm{Q}_{L^2}\norm{R}_{L^2}.
	\]
	Moreover, we have the identities
	\[
	\norm{P}_{L^2}=\norm{p}_2=\sqrt{2-\beta^2},
	\qquad
	\norm{R}_{L^2}=\norm{r}_2=\alpha,
	\]
	\[
	\norm{Q}_{L^2}=\norm{q}_2=\beta,
	\qquad
	\norm{S}_{L^2}=\norm{s}_2=\sqrt{2-\alpha^2}.
	\]
	Now assume, for the sake of contradiction, that $\eps<c_{\mathrm{line}}(\delta,\eta)$, where
	$c_{\mathrm{line}}(\delta,\eta)$ is chosen so small that
	\[
	\sqrt2\,d^*(\delta,\eta)K_{\delta,\eta}c_{\mathrm{line}}(\delta,\eta)\le1.
	\]
	Under these assumptions, we may conclude that $\alpha\le1$ and
	\[
    1\le \min\{\|P\|_{L^2},\|S\|_{L^2}\}\le \max\{\|P\|_{L^2},\|S\|_{L^2}\} 
	\le \sqrt{2}.
	\]
	Consequently, from \eqref{eq:aux_conc_branch} we may bound
	\begin{equation}
    \label{eq:aux2_conc_branch}
	\E[\abs{P}]\,\E[\abs{S}]
	\le
	\eps+\sqrt2\,\alpha+\sqrt2\,\beta+\alpha\beta.
	\end{equation}
	To lower bound the left-hand side of \eqref{eq:aux2_conc_branch}, we decompose $P$ as 
	\[
	P=p_1\eta+G,
	\qquad \text{where} \quad 
	G:=\sum_{i\in H\cap\{2,3,\dots\}} p_i\xi_i.
	\]
	Note that either $\abs{p_1}\ge \norm{p}_2/\sqrt2$ or $\norm{G}_{L^2}\ge \norm{p}_2/\sqrt2$.
	In the first case, we apply Jensen's inequality to estimate
	\[
	\E[\abs{P}]
	=\E\bigl[\abs{p_1\eta+G}\bigr]
	\ge
	\abs{p_1}\,\E[\abs{\eta}]
	\ge
	\kappa_{\delta,\eta}\norm{p}_2.
	\]
	In the second case, we let $G'$ be an independent copy of $G$ and compute that
	\[
	2\E[\abs{P}]
	=\E\bigl[\abs{p_1\eta+G}+\abs{p_1\eta+G'}\bigr]
	\ge \E[\abs{G-G'}].
	\]
	Then, we write
	\[
	G-G'
	=
	\sum_{i\in H\cap\{2,3,\dots\}} p_i(\xi_i-\xi_i'),
	\]
	where $(\xi_i')$ is an independent copy of $(\xi_i)$.  The variables
	$\xi_i-\xi_i'$ are independent, centered, satisfy
	\[
	\norm{\xi_i-\xi_i'}_{L^2}=\sqrt2,
	\qquad
	\norm{\xi_i-\xi_i'}_{L^1}\ge \delta,
	\]
	and therefore also satisfy
	\[
	\norm{\xi_i-\xi_i'}_{L^1}
	\ge \frac{\delta}{\sqrt2}\norm{\xi_i-\xi_i'}_{L^2}.
	\]
	Applying Lemma~\ref{lem:l1lower}, we may bound
	\[
	\E[\abs{G-G'}]
	\ge
	\frac{\delta}{2\sqrt2}
	\Bigl(\sum_{i\in H\cap\{2,3,\dots\}} p_i^2
	\norm{\xi_i-\xi_i'}_{L^2}^2\Bigr)^{1/2}
	=\frac{\delta}{2}\norm{G}_{L^2}
	\ge 2\kappa_{\delta,\eta}\norm{p}_2.
	\]
	In both cases, we obtain
	\[
	\E[\abs{P}]
	\ge \kappa_{\delta,\eta}\norm{p}_2
	\ge \kappa_{\delta,\eta}.
	\]
	Since $S$ is supported on $H^c$, we may apply
	Lemma~\ref{lem:l1lower} to obtain
	\[
	\E[\abs{S}]
	\ge \frac{\delta}{2\sqrt2}\norm{s}_2
	\ge \frac{\delta}{2\sqrt2}.
	\]
	Combining the last two displays with the inequality in \eqref{eq:aux2_conc_branch}, we see that
	\[
	\frac{\delta\kappa_{\delta,\eta}}{2\sqrt2}
	\le
	\eps+\sqrt2\,\alpha+\sqrt2\,\beta+\alpha\beta.
	\]
	Using the bounds for $\alpha$ and $\beta$, we infer that
	\[
	\frac{\delta\kappa_{\delta,\eta}}{2\sqrt2}
	\le
	\Bigl(1+2d^*(\delta,\eta)K_{\delta,\eta}
	+\frac{\theta}{16}d^*(\delta,\eta)K_{\delta,\eta}\Bigr)\eps
	+\frac{\sqrt2\,\theta}{32}.
	\]
	By the choice of $\theta$, we deduce the inequality
	\[
	\frac{\sqrt2\,\theta}{32}
	\le \frac{\delta\kappa_{\delta,\eta}}{8\sqrt2}.
	\]
	After decreasing $c_{\mathrm{line}}(\delta,\eta)$ if necessary so that
	\[
	\Bigl(1+2d^*(\delta,\eta)K_{\delta,\eta}
	+\frac{\theta}{16}d^*(\delta,\eta)K_{\delta,\eta}\Bigr)
	c_{\mathrm{line}}(\delta,\eta)
	\le \frac{\delta\kappa_{\delta,\eta}}{8\sqrt2}
	\]
	(which is strictly smaller than
$\delta\kappa_{\delta,\eta}/(2\sqrt2)$ whenever
	$\eps<c_{\mathrm{line}}(\delta,\eta)$), we reach a contradiction. We conclude, therefore, that
	$\eps\ge c_{\mathrm{line}}(\delta,\eta)$.
\vspace{0.5em}

\noindent\emph{Case 2: $\eta$ is Rademacher and $H=\{1\}$.}
Write
\[
u:=a\mathbf 1_{H^c},
\qquad
v:=b\mathbf 1_{H^c},
\qquad
\rho:=\theta/32,
\]
and let
\[
U:=\sum_{i\ge2} a_i\xi_i,
\qquad
V:=\sum_{i\ge2} b_i\xi_i.
\]
Since $H=\{1\}$, every coordinate entering $U$ and $V$ is good.  The
 assumption of being concentrated states precisely that
\[
\norm{u-\tau v}_2=\norm{U-\tau V}_{L^2}\le \rho.
\]
Set
\[
R:=U-\tau V,
\qquad
S:=U+\tau V,
\qquad
\alpha:=a_1-\tau b_1,
\qquad
\beta:=a_1+\tau b_1
\]
and decompose
\[
X-\tau Y=\alpha\eta+R,
\qquad
X+\tau Y=\beta\eta+S.
\]
Since $\norm{a}_2=\norm{b}_2=1$ and $\ip{a}{b}=0$, we have
\[
\norm{a-\tau b}_2^2=\norm{a+\tau b}_2^2=2.
\]
Therefore,
\[
\alpha^2+\norm{R}_{L^2}^2
=\alpha^2+\norm{u-\tau v}_2^2
=2,
\qquad
\beta^2+\norm{S}_{L^2}^2
=\beta^2+\norm{u+\tau v}_2^2
=2.
\]
In particular,
\[
|\alpha|\ge \sqrt{2-\rho^2}\ge1.
\]
Our next objective is to estimate $\beta$. To this end, we notice that
\[
|\beta|\le |\alpha\beta|
=|a_1^2-b_1^2|
=\bigl|\norm{v}_2^2-\norm{u}_2^2\bigr|
=\bigl|\langle u-\tau v,u+\tau v\rangle\bigr|.
\]
Thus, it follows that
\[
|\beta|
\le \norm{u-\tau v}_2\norm{u+\tau v}_2
\le \sqrt2\,\rho.
\]
Hence, $\norm{S}_{L^2}^2 = \|u+\tau v\|_2^2\ge1$ and Lemma~\ref{lem:l1lower} gives
\[
\norm{S}_{L^1}
=\Bigl\|\sum_{i\ge2}(a_i+\tau b_i)\xi_i\Bigr\|_{L^1}
\ge \frac{\delta}{2\sqrt2}\norm{u+\tau v}_2
\ge \frac{\delta}{2\sqrt2}.
\]
On the other hand, we have
\[
\norm{R}_{L^1}\le \norm{R}_{L^2}\le \rho.
\]
Now, we condition on the $\sigma$-algebra generated by the good coordinates.  Since
$\eta$ is an independent Rademacher variable, we have
\[
X^2-Y^2
=(X-\tau Y)(X+\tau Y)
=(\alpha\eta+R)(\beta\eta+S)
=\alpha\beta+RS+\eta(\alpha S+\beta R).
\]
Therefore,
\begin{align*}
    \E\bigl[\abs{X^2-Y^2}\mid R,S\bigr]
&=\frac12\Bigl(
\abs{\alpha\beta+RS+\alpha S+\beta R}
+
\abs{\alpha\beta+RS-\alpha S-\beta R}
\Bigr)
\\
&=\max\bigl(\abs{\alpha\beta+RS},\abs{\alpha S+\beta R}\bigr)
\ge \abs{\alpha S+\beta R}.
\end{align*}
Taking expectations on both sides, by the law of iterated expectation, we have 
\begin{equation*}
\E[|X^2-Y^2|]\ge \|\alpha S+\beta R\|_{L^1}\ge  |\alpha|\|S\|_{L^1}-|\beta|\|R\|_{L^1}.
\end{equation*}
Recalling the definition of $\varepsilon$ and the estimates for the $L^1$ norms of $S$ and $R$, we conclude that
\[
\eps
\ge |\alpha|\norm{S}_{L^1}-|\beta|\norm{R}_{L^1}
\ge \frac{\delta}{2\sqrt2}-\sqrt2\,\rho^2.
\]
Since $\rho=\theta/32\le \delta^2/256$ and $0<\delta\le1$, the right-hand side
is bounded below by a positive constant depending only on $\delta$ (for
instance by $\delta/4$).  After decreasing
$c_{\mathrm{line}}(\delta,\eta)$ if necessary, this proves Case~2.
\end{proof}

\subsection{The diffuse case}
We now move to the complementary regime.  Here, both $a-b$ and $a+b$ retain a
definite amount of $\ell^2$ mass on the tail, and every tail coordinate is
individually small.  This is exactly the setting in which an 
anticoncentration estimate may be used.
We set
\[
c:=\frac{(a-b)\mathbf 1_{H^c}}{\sqrt2},
\qquad
d:=\frac{(a+b)\mathbf 1_{H^c}}{\sqrt2}.
\]
Since $1\in H$, the vectors $c$ and $d$ are supported only on good
coordinates.  In the diffuse case, both have $\ell^2$ norm at least
$\sqrt{8s_0}$, while each coordinate is bounded by $\sqrt2\,\lambda$.
Our main anticoncentration lemma is as follows.

\begin{lemma}\label{lem:cross}
There exists $r_*=r_*(\delta,\eta)>0$ such that, in the diffuse case,
\[
\Prob\bigl(\min(\abs{X_c+u},\abs{X_d+v})>r_*\bigr)\ge \frac14, \quad 
\qquad(u,v\in\R).
\]
\end{lemma}

\begin{proof}
We first prove a small-ball estimate.  Let $e\in \ell^2(\N)$ be
supported on $H^c$ and satisfy
\[
\norm{e}_2^2\ge 8s_0
\qquad \text{and} \quad
\sup_{i\ge2}\abs{e_i}\le \sqrt2\,\lambda.
\]
	We claim that there exists $r_0=r_0(\delta,\eta)>0$ such that
\begin{equation}\label{eq:onevar-smallball}
\Prob\bigl(\abs{X_e-t}\le r_0\bigr)\le \frac38
\qquad(t\in\R).
\end{equation}
Once this is proved, the lemma follows immediately by applying
\eqref{eq:onevar-smallball} to $e=c$ with $t=-u$ and to $e=d$ with $t=-v$.
Indeed, in the diffuse case both $c$ and $d$ satisfy the displayed
hypotheses, so
\[
\Prob\bigl(\abs{X_c+u}\le r_0\bigr)\le \frac38,
\qquad
\Prob\bigl(\abs{X_d+v}\le r_0\bigr)\le \frac38.
\]
By the union bound, we have
\[
\Prob\bigl(\abs{X_c+u}\le r_0\ \text{or}\ \abs{X_d+v}\le r_0\bigr)
\le \frac34,
\]
and therefore
\[
\Prob\bigl(\min(\abs{X_c+u},\abs{X_d+v})>r_0\bigr)\ge \frac14.
\]
Thus, it remains to establish the claim in \eqref{eq:onevar-smallball}. 
Set
\[
M:=\Bigl\lceil 2^{16}\delta^{-2}\Bigr\rceil,
\qquad
\eta:=\frac{s_0}{2M},
\qquad
r_0:=\frac{\delta\sqrt{\eta}}{16\sqrt{2}}.
\]
	As
	\[
	2\lambda^2
	\le \frac{\delta^2\theta^2}{2^{33}}
	\le \frac{\theta^2}{2^{15}M}
	=\eta,
	\]
	every coordinate satisfies $e_i^2\le \eta$.  Since
\[
\norm{e}_2^2\ge 8s_0>2M\eta=s_0,
\]
we may choose pairwise disjoint finite sets $G_1,\dots,G_M\subset H^c$ such
that
\begin{equation}\label{eq:block-energy}
\eta\le \sum_{i\in G_j} e_i^2\le 2\eta
\qquad(1\le j\le M).
\end{equation}
Indeed, we may construct the blocks greedily:  Suppose that $G_1,\dots,G_{j-1}$ have
already been chosen.  Their total mass is at most $2(j-1)\eta$, so the
unused part of $e$ still has mass at least
\[
8s_0-2(j-1)\eta\ge 8s_0-2M\eta=7s_0> \eta.
\]
Hence, we may choose a finite subset of the remaining support whose mass is
at least $\eta$, and by minimality its mass is at most $\eta+\sup_i e_i^2
\le 2\eta$.  This produces the next block.
 
 Let $X_e'$ be an independent copy of $X_e$, realized on a second copy of the
probability space.  For each $i\ge2$ define
\[
D_i:=\xi_i(\omega_1)-\xi_i'(\omega_2).
\]
The variables $(D_i)_{i\ge2}$ are independent, centered and symmetric,
with
\[
\norm{D_i}_{L^2}^2
=\E[(\xi_i-\xi_i')^2]
=2.
\]
Moreover, by Jensen's inequality applied conditionally on $\xi_i$, we see that
\[
\E[\abs{D_i}]
=\E\bigl[\E[\abs{\xi_i-\xi_i'}\mid \xi_i]\bigr]
\ge \E\bigl[\abs{\xi_i-\E[\xi_i']} \bigr]
=\E[\abs{\xi_i}]
\ge \delta.
\]
Hence,
\[
\E[\abs{D_i}]
\ge \frac{\delta}{\sqrt2}\norm{D_i}_{L^2}.
\]
For each block, we define
\[
Y_j:=\sum_{i\in G_j} e_iD_i,
\qquad
\sigma_j^2:=\sum_{i\in G_j} e_i^2.
\]
By \eqref{eq:block-energy}, we have
\[
\eta\le \sigma_j^2\le 2\eta.
\]
Applying Lemma~\ref{lem:l1lower} to the independent centered family
$(D_i)_{i\in G_j}$, with lower $L^1$ ratio $\delta/\sqrt2$, gives
\[
\E[\abs{Y_j}]
\ge
\frac{(\delta/\sqrt2)}{2\sqrt2}
\Bigl(\sum_{i\in G_j} e_i^2\norm{D_i}_{L^2}^2\Bigr)^{1/2}
=\frac{\delta}{4}\bigl(2\sigma_j^2\bigr)^{1/2}
=\frac{\delta}{2\sqrt{2}}\sigma_j
\ge \frac{\delta\sqrt{\eta}}{2\sqrt{2}}.
\]
Next, we set
\[
a_0:=\frac{\delta\sqrt{\eta}}{4\sqrt{2}}=4r_0,
\]
and notice that $a_0\le \E[\abs{Y_j}]/2$.  We define the event 
\[
A_j:=\{\abs{Y_j}\ge a_0\}.
\]
Since
\[
\E[Y_j^2]=2\sigma_j^2\le 4\eta,
\]
we obtain
\[
\E[\abs{Y_j}]
\le a_0+\E[\abs{Y_j}\mathbf 1_{A_j}]
\le a_0+\Prob(A_j)^{1/2}\E[Y_j^2]^{1/2}\le a_0 +\Prob(A_j)^{1/2}(4\eta)^{1/2} .
\]
As $\E[\abs{Y_j}]\ge 2a_0$, the second term on the right-hand side satisfies
\[
\Prob(A_j)^{1/2}\,(4\eta)^{1/2}\ge a_0.
\]
Therefore,
\begin{equation}\label{eq:block-active}
\Prob(A_j)\ge \frac{a_0^2}{4\eta}=\frac{\delta^2}{128}.
\end{equation}
Now, we let
\[
N:=\sum_{j=1}^M \mathbf 1_{A_j}
\]
and notice that this is a sum of independent random variables because the blocks are
disjoint.  By \eqref{eq:block-active}, we have
\[
\E[N]
=\sum_{j=1}^M \Prob(A_j)
\ge M\frac{\delta^2}{128}
\ge 512.
\]
Since $\operatorname{Var}(\mathbf 1_{A_j})\le \E[\mathbf 1_{A_j}]$, independence implies that
\[
\operatorname{Var}(N) = \sum_{j=1}^M \operatorname{Var}(\mathbf 1_{A_j})\le \sum_{j=1}^M \E[\mathbf 1_{A_j}]=  \E[N].
\]
Applying Chebyshev's inequality, we see that
\[
\Prob(N<256)
\le
\Prob\bigl(\abs{N-\E[N]}\ge \E[N]-256\bigr)
\le
\frac{\operatorname{Var}(N)}{(\E[N]-256)^2}
\le
\frac{\E[N]}{(\E[N]-256)^2}
\le \frac{512}{256^2}
<\frac1{16}.
\]
 Next, we aim to achieve a Sperner bound for the symmetrized sum.
Let
\[
U:=\sum_{j=1}^M Y_j.
\]
We claim that for every $s\in\R$, we have
\begin{equation}\label{eq:symm-smallball}
\Prob(\abs{U-s}\le 2r_0)\le \frac18.
\end{equation}
We begin by decomposing
\[
\Prob(\abs{U-s}\le 2r_0)
\le
\Prob(\abs{U-s}\le 2r_0,\ N\ge256)+\Prob(N<256).
\]
We have already bounded the second term by $1/16$, so it remains to estimate
the first one.
To achieve this, we condition on the absolute values $(\abs{Y_j})_{j=1}^M$.  Since each $Y_j$ is
symmetric, conditional on $(\abs{Y_j})_j$ the signs of the nonzero $Y_j$ are
independent Rademacher variables.  On the event $\{N\ge256\}$, let
\[
J:=\{j:\abs{Y_j}\ge a_0\},
\qquad
m:=\#J\ge256.
\]
After freezing the inactive coordinates and the magnitudes $\abs{Y_j}$, the
condition $\abs{U-s}\le 2r_0$ becomes
\[
\sum_{j\in J}\varepsilon_j w_j \in I,
\]
where $w_j:=\abs{Y_j}\ge a_0$ and $I$ is an interval of diameter
$4r_0=a_0$.  Since $a_0<2w_j$ for every $j\in J$, the family of sign patterns
for which the weighted sum falls in $I$ is an antichain in $\{\pm1\}^J$.
Indeed, if two such sign patterns were comparable in the natural subset order,
their sums would differ by at least $2w_j\ge2a_0>a_0$, which is impossible
inside an interval of diameter $a_0$.  By Sperner's theorem, the conditional
probability of this event is therefore at most
\[
\frac{\binom{m}{\lfloor m/2\rfloor}}{2^m}
\le \frac1{\sqrt m}
\le \frac1{16},
\]
because $m\ge256$.  Averaging over the conditioning yields
\[
\Prob(\abs{U-s}\le 2r_0,\ N\ge256)\le \frac1{16}.
\]
Together with the bound for $\Prob(N<256)$, this proves
\eqref{eq:symm-smallball}.

 Finally, let
\[
R:=\sum_{i\notin G_1\cup\cdots\cup G_M} e_iD_i
\]
and notice that $R$ is independent of $U$, because it is built from coordinates disjoint
from those entering the block sums.  Also, by construction, we have
\[
X_e(\omega_1)-X_e'(\omega_2)=U+R.
\]
Fix $t\in\R$.  If both $\abs{X_e(\omega_1)-t}\le r_0$ and
$\abs{X_e'(\omega_2)-t}\le r_0$, then
\[
\abs{X_e(\omega_1)-X_e'(\omega_2)}\le 2r_0.
\]
Hence, since $X_{e}'$ is an independent copy of $X_e$, we have
\[
\Prob(\abs{X_e-t}\le r_0)^2 = \Prob(\abs{X_e-t}\le r_0\cap\abs{X_e'-t}\le r_0)
\le
\Prob(\abs{X_e(\omega_1)-X_e'(\omega_2)}\le 2r_0)
=
\Prob(\abs{U+R}\le 2r_0).
\]
Conditioning on $R$ and using \eqref{eq:symm-smallball}, we deduce that
\[
\Prob(\abs{U+R}\le 2r_0)
\le
\sup_{s\in\R}\Prob(\abs{U-s}\le 2r_0)
\le \frac18.
\]
Therefore,
\[
\Prob(\abs{X_e-t}\le r_0)\le \frac1{\sqrt8}<\frac38.
\]
This proves \eqref{eq:onevar-smallball}, and with it the lemma.
\end{proof}

Once the above anticoncentration estimate is available, the lower bound for
$\eps$ follows from the same factorization of $X^2-Y^2$ that was
 used in the compactness proof.

\begin{proposition}[Diffuse branch]\label{prop:diffuse}
If the diffuse alternative of Lemma~\ref{lem:dichotomy} holds, then
\[
\eps\ge c_{\mathrm{cross}}(\delta,\eta)>0.
\]
\end{proposition}

\begin{proof}
Let
\[
h_\pm:=X_{(a\pm b)\mathbf 1_H}.
\]
Then
\[
X^2-Y^2
=2\bigl(X_d+h_+/\sqrt2\bigr)\bigl(X_c+h_-/\sqrt2\bigr),
\]
and therefore
\[
\abs{X^2-Y^2}
\ge
2\min\Bigl(\abs{X_d+h_+/\sqrt2},\abs{X_c+h_-/\sqrt2}\Bigr)^2.
\]
We condition on the head variables $h_\pm$ and apply Lemma~\ref{lem:cross} with
the shifts $u=h_-/\sqrt2$ and $v=h_+/\sqrt2$.  This yields a uniform positive
lower bound for the conditional expectation of the squared minimum, hence for
$\eps$.  The resulting constant depends only on $\delta$ and $\eta$.
\end{proof}

\subsection{Conclusion of the second proof}
At this point, the proof is immediate. Indeed, the above dichotomy shows that every
orthogonal pair falls into exactly one of two cases, and in each case we have
achieved a quantitative lower bound on $\eps$.

\begin{proof}[Second proof of Theorem~\ref{thm:main}]
By Proposition~\ref{prop:orth-reduction}, it is enough to prove a uniform
 lower bound for orthogonal unit vectors.  Let $a,b$ be orthogonal unit vectors and form $X$
and $Y$ as above.  Applying the dichotomy in Lemma~\ref{lem:dichotomy}, we reduce to two possible scenarios. 
If the concentrated alternative holds, then
Proposition~\ref{prop:concentrated} yields
\[
\eps\ge c_{\mathrm{line}}(\delta,\eta).
\]
If, instead, the diffuse alternative holds, then Proposition~\ref{prop:diffuse} yields
\[
\eps\ge c_{\mathrm{cross}}(\delta,\eta).
\]
Since $\eps\le 2\Delta$, we conclude that
\[
\Delta
=\norm{|X|-|Y|}_{L^2}
\ge
\frac12\min\bigl(c_{\mathrm{line}}(\delta,\eta),c_{\mathrm{cross}}(\delta,\eta)\bigr)
=:c_{\delta,\eta}>0.
\]
The orthogonal reduction then gives
\[
\min_{\sigma\in\{\pm1\}}\norm{X-\sigma Y}_{L^2}
\le C_{\delta,\eta}\norm{|X|-|Y|}_{L^2}
\]
for all $X,Y\in \overline{\Span}(\eta,\xi_i:i\ge2)$, completing the proof of
Theorem~\ref{thm:main}.
\end{proof}

\section{Comments and open problems}\label{sec:comments}

The above proofs suggest several natural directions for
further work.  Although the theorem proved here characterizes when the closed span of mean-zero independent random variables does stable phase retrieval in $L^2(\Omega;\mathbb{R})$, the mechanisms that enter
the proof are likely flexible enough to tackle some adjacent
questions. 
\begin{question}
Real phase retrieval concerns recovery from the absolute value modulo the symmetry $\{\pm1\}$.  It is
natural to ask whether there are analogous theorems for other nonlinearities modulo other symmetry groups.  
\end{question}
The qualitative arguments in this paper are likely to be useful for pursuing the above direction, since they separate the
problem into a limit identity and a geometric classification of the
corresponding support set.  We believe that there is likely a broader geometric framework where our argument can be generalized,
which we hope also covers complex phase retrieval and declipping.

Our next question concerns the extent to which the present theorem depends fundamentally on the norm under consideration.

\begin{question}
Which analogues of the present theorem remain true when
the ambient space is no longer $L^2$? In particular, is there an $L^p$-analogue of our result when $p<2$? The answer is yes when $p\geq 2$,
as an immediate consequence of the interpolation/extrapolation theory developed in \cite{FOTP}.
\end{question}
In \cite[Section 4]{FOTP} it is shown that for certain random variables $f$ on $(0,1)$, the span of independent 
copies of $f$ will do stable phase retrieval in any r.i.~function space $X$ on $(0,1)$ for which $\norm{f}_X<\infty.$
We do not know to what extent this is a general phenomenon of i.i.d.~random variables or if the norm plays a more fundamental role in the
stability problem for independent random variables.

Our final question is motivated by the many examples of sequences of functions (for example, sparse Fourier series) that behave in certain senses like independent random variables. It would be interesting to know to what extent the methods developed in this paper can be generalized to prove stable phase retrieval for such subspaces.
\begin{question}
To what extent can independence be weakened in Theorem~\ref{thm:main}?
\end{question}

\section{Discussion on the Lean verification}
We verified the statement of Theorem~\ref{thm:lean} using the Lean 4 theorem prover \cite{moura2021lean}. The formalization can be found  at {\small \url{https://github.com/jaumededios/RandomVarSPR}}. Note the slight change of notation in the indexing
  (as in Lean $\mathbb{N}=\{0,1,2,\ldots\}$) as well as the uniform constants, which we now wish to track more precisely.
\begin{theorem} \label{thm:lean}
    Let $\Omega$ be a probability space and let $\xi_0,\xi_1,\ldots, \xi_{n}, \ldots \in L^2(\Omega;\mathbb{R})$ be independent random variables satisfying
    \[
    \forall i \ge 0,\quad \|\xi_i\|_2 = 1, \quad \mathbb E[\xi_i] = 0 \quad  \text{and} \quad \E[|\xi_i|] \ge A>0,
    \]
    \[
    \forall i  \ge 1, \quad \E[|\xi_i|] \le B<1 .
    \]
    Let $X, Y \in \overline{span}(\{\xi_0,\xi_1, \dots \})\subset L^2(\Omega)$.
    Then 
    \[
    \min_{\tau = \pm 1} \|X-\tau Y\|_{L^2} \le C_{A,B}\||X|-|Y|\|_{L^2}
    \]
     for some constant $C_{A,B}$ satisfying
    \[
    C_{A,B}
    \le
    2^{64}
    \max \left(A^{-10},A^{-8}/(1-B)\right).
    \]
\end{theorem}
The file {\scriptsize \texttt{showcase.lean}} provides a self-contained version of the main statement that was Lean verified, which is a direct translation of Theorem~\ref{thm:lean}. Following \cite{bertolini20262}, our goal when writing {\scriptsize \texttt{showcase.lean}} was to maximize intelligibility for mathematicians not well-versed in Lean, even
if it sometimes resulted in less idiomatic Lean code. The Lean statements were carefully reviewed and written by the authors. The Lean proof, on the other hand, was mostly translated from our natural language proof using LLMs, specifically GPT 5.4 through Codex and Claude Opus 4.6 through Claude Code, using the lean-lsp-mcp \cite{lean-lsp-mcp} and custom skills.

Reviewing the statements in {\scriptsize \texttt{showcase.lean}} is enough to attest to the semantic correctness of the formalized statement, conditional on the semantics of the Mathlib library itself\footnote{In particular, our formalization, and the trust one can place it, is possible only because of the enormous
efforts by the Mathlib community \cite{mathlib2020} in carefully formalizing mathematical objects in Lean and
ensuring that their meaning is what the mathematical community expects.}. Moreover, the statements in this file are simple and we hope that, at a broad level, they can be understood by mathematicians without Lean knowledge. The objective of this section is to explain the contents of {\scriptsize \texttt{showcase.lean}}. The reader may also consult \cite{armstrong2026formalization,chen2026thakur,hariharan2026milestone,ho2026erd,ilin2026semi} for a selection of recent examples of formalizations involving AI. 

We start by defining $\Omega$ as a probability space. In Mathlib, this is expressed as
\begin{leancode}
variable {Ω : Type*} 
         [MeasureSpace Ω] 
         [IsProbabilityMeasure (ℙ : Measure Ω)].
\end{leancode}
The first line defines $\Omega$ as a type (which one could think of as a set), the second line states that this type $\Omega$ is a measure space, that is, a set with a \textit{canonical} measure \lean{ℙ} and associated $\sigma$-algebra. The last line states that the canonical measure \lean{ℙ} on $\Omega$ is a probability measure (i.e.~that its volume is 1). We must write \lean{(ℙ : Measure Ω)} because Lean is not able to automatically infer that \lean{ℙ} refers to the probability measure on $\Omega$ unless we explicitly state so.

We now define a family of independent random variables in $L^2(\Omega)$. In Lean, a countable family of random variables in $L^2(\Omega)$ is a function from $\mathbb N$ to $L^2(\Omega)$ with the property that \emph{the function} is independent.
\begin{leancode}
variable (ξ : ℕ → Lp ℝ 2 (ℙ : Measure Ω))
         (hIndep : iIndepFun (fun i => ξ i) ℙ)
         (hMean : ∀ i, !$\mathbb E$![ξ i] = 0)
         (hVar : ∀ i, !$\mathbb E$![(ξ i : Ω → ℝ) ^ 2] = 1)
\end{leancode}
Random variables are functions, and Mathlib uses the notation of independent \emph{functions} instead of independent random variables. \lean{iIndepFun} defines the independence of an indexed family of random variables (if one was interested in two random variables instead of an indexed family, one would use \lean{IndepFun}). The \lean{iIndepFun} statement takes two inputs:
    \begin{enumerate}
        \item The family of random variables (in this case $i\mapsto \xi_i$).
        \item The probability measure (in this case \lean{ℙ}).
    \end{enumerate}
  In the third and fourth lines, \lean{hMean} and \lean{hVar} state the hypotheses that each random variable is mean zero and of variance one. 
Note that the fourth line has a type coercion of the form \lean{(ξ i : Ω → ℝ)}. This means that we consider $\xi_i$ as a function from $\Omega$ to $\mathbb R$. Originally, $\xi_i$ is an element of $L^2(\Omega)$, and this command coerces it into a particular function. This is needed because Mathlib does not, by default, have a square operation from $L^2$ to $L^1$, but squaring is well-defined for functions.

At this point, we can define the upper bound for the stability constant as a function of $A,B$:
\begin{leancode}
def StabilityConstant (A B : ℝ) : ℝ :=
  2 ^ 64 * max (A⁻¹ ^ 10) (A⁻¹ ^ 8 / (1 - B))
\end{leancode}
Theorem~\ref{thm:lean} proves stability for all functions in the topological closure of the span of the $\xi_i$ variables, which can be defined in Lean as
\begin{leancode}
def closedSpan (ξ : ℕ → Lp ℝ 2 (ℙ : Measure Ω)) :
    Submodule ℝ (Lp ℝ 2 (ℙ : Measure Ω)) :=
  (Submodule.span ℝ (Set.range ξ)).topologicalClosure
\end{leancode}
The function \lean{closedSpan} takes a sequence of random variables in $L^2(\Omega)$ and returns a vector subspace of $L^2(\Omega)$; namely, the topological closure of the span of the range of the map $i\mapsto \xi_i$. Lean defines vector spaces as \textit{modules}, as the fact that the base ring happens to be a field may be irrelevant for some statements.  With these definitions in hand, we can now define the main statement, which is equivalent to Theorem~\ref{thm:lean}. We note that the variables $\xi_i$, together with their hypotheses, are globally defined in the file and can be used in the statement of \lean{quantitative_stable_phase_retrieval_L2} without explicit redefinition.
\begin{leancode}
theorem quantitative_stable_phase_retrieval_L2
    {A B : ℝ} (hA : 0 < A) (hB : B < 1)
    (hL1_lower : ∀ i, A ≤ !$\mathbb E$![|ξ i|])
    (hL1_upper : ∀ i, i > 0 → !$\mathbb E$![|ξ i|] ≤ B)
    {X Y : Lp ℝ 2 (ℙ : Measure Ω)}
    (hX : X ∈ closedSpan ξ) (hY : Y ∈ closedSpan ξ)
    : 
    min ‖X - Y‖ ‖X + Y‖ ≤ StabilityConstant A B * ‖|X| - |Y|‖
\end{leancode}

\bibliographystyle{plain}
\bibliography{refs}

\end{document}